\documentclass[english,british, toc=bibliography, toc=index, hidelinks, DIV=12,abstract=true]{scrartcl}
\usepackage[T1]{fontenc}
\usepackage[latin9]{inputenc}

\usepackage{babel}
\usepackage{varioref}
\usepackage{prettyref}
\usepackage{float}
\usepackage{mathtools}
\usepackage{enumitem}
\usepackage{multirow}
\usepackage{amsmath}
\usepackage{amsthm}
\usepackage{amssymb}
\usepackage{graphicx}
\usepackage[numbers,longnamesfirst]{natbib}

\makeatletter

\floatstyle{ruled}
\newfloat{algorithm}{tbp}{loa}
\providecommand{\algorithmname}{Algorithm}
\floatname{algorithm}{\protect\algorithmname}

\newtheorem{definition}{Definition}%

\theoremstyle{plain}
\newtheorem{thm}{\protect\theoremname}
\theoremstyle{plain}
\newtheorem{prop}[thm]{\protect\propositionname}
\theoremstyle{remark}
\newtheorem{rem}[thm]{\protect\remarkname}
\theoremstyle{plain}

\theoremstyle{plain}

\def\documenttitle{Lost customer approximations of SOQN with application to RMFS}

\usepackage{mathalfa}
\usepackage{lscape}
\usepackage{subfigure}
\usepackage{relsize}
\usepackage{listings} %R-Code einbinden.
\usepackage{booktabs} % Bessere Tabellen.
\usepackage{array}
\usepackage{siunitx}
\usepackage[titletoc,title]{appendix}
\usepackage{algorithm}
\usepackage{algpseudocode}
\usepackage{authblk} % for author affiliation

\usepackage[unicode=true,
bookmarks=true,bookmarksnumbered=false,bookmarksopen=false,
breaklinks=true,pdfborder={0 0 0},backref=false,colorlinks=false]   %=false
{hyperref}

\hypersetup{
	linkcolor=blue,     %linkcolor=blue, 
	citecolor=blue, urlcolor=blue, filecolor=blue,
	pdftitle={\documenttitle},
	pdfauthor={Sonja Otten},
	pdfkeywords={}
}

\newrefformat{def}{Definition \ref{#1}}
\newrefformat{rem}{Remark \ref{#1}}
\newrefformat{sect}{Section \ref{#1}}
\newrefformat{sub}{Section \ref{#1}}
\newrefformat{prop}{Proposition \ref{#1}}
\newrefformat{thm}{Theorem \ref{#1}}
\newrefformat{chap}{Chapter \ref{#1}}
\newrefformat{cor}{Corollary \ref{#1}}
\newrefformat{par}{Paragraph \ref{#1}}
\newrefformat{ex}{Example \ref{#1}}
\newrefformat{part}{Part \ref{#1}}
\newrefformat{fig}{Figure \ref{#1}}
\newrefformat{app}{Appendix \ref{#1}}

\makeatother

\addto\captionsbritish{\renewcommand{\corollaryname}{Corollary}}
\addto\captionsbritish{\renewcommand{\lemmaname}{Lemma}}
\addto\captionsbritish{\renewcommand{\propositionname}{Proposition}}
\addto\captionsbritish{\renewcommand{\remarkname}{Remark}}
\addto\captionsbritish{\renewcommand{\theoremname}{Theorem}}
\addto\captionsenglish{\renewcommand{\algorithmname}{Algorithm}}
\addto\captionsenglish{\renewcommand{\corollaryname}{Corollary}}
\addto\captionsenglish{\renewcommand{\lemmaname}{Lemma}}
\addto\captionsenglish{\renewcommand{\propositionname}{Proposition}}
\addto\captionsenglish{\renewcommand{\remarkname}{Remark}}
\addto\captionsenglish{\renewcommand{\theoremname}{Theorem}}
\providecommand{\corollaryname}{Corollary}
\providecommand{\lemmaname}{Lemma}
\providecommand{\propositionname}{Proposition}
\providecommand{\remarkname}{Remark}
\providecommand{\theoremname}{Theorem}

\begin{document}
	\global\long\def\myV{V}%
	\global\long\def\myv{v}%
	\global\long\def\evect{\mathbf{e}}%
	\global\long\def\maxdiff{\widetilde{\max}(\mathbf{k})}%
	\global\long\def\routingprob{p_{i}(\mathbf{k})}%
	
	\global\long\def\kvect{\mathbf{k}}%
	\global\long\def\mvect{\mathbf{m}}%
	\global\long\def\nvect{\mathbf{n}}%
	\global\long\def\pvect{\mathbf{p}}%
	\global\long\def\zvect{\mathbf{z}}%
	\global\long\def\bvect{\mathbf{b}}%
	\global\long\def\routep#1#2{r\left(#1,#2\right)}%
	\global\long\def\Mset{\overline{M}}%
	\global\long\def\Jset{\overline{J}}%
	\global\long\def\supplierrate{\nu}%
	\global\long\def\phantomeq{\mathrel{\phantom{=}}}%
	
\title{Load balancing in a network of queueing-inventory systems}

%\date{3nd December 2019}
\date{}
% Use small font for affiliations.
\renewcommand\Affilfont{\small \itshape}
\author{Sonja Otten} 
\affil{Hamburg University of Technology, Am Schwarzenberg-Campus 3, 21073 Hamburg}
\renewcommand\Authands{, }

\maketitle
\vspace{-1.0cm}
\begin{abstract}
	\noindent We study a supply chain consisting of production-inventory systems at several locations which are coupled by a common supplier. Demand of customers arrives at each production system according to a Poisson process and is lost if the local inventory is depleted (`lost sales'). To satisfy a customer's demand a server at the production system needs raw material from the associated local inventory. The supplier manufactures raw material to replenish the local inventories, which are controlled by a continuous review base stock policy. The routing of items depends on the on-hand inventory at the locations with the aim to obtain `load balancing'.
	We show that the stationary distribution has a product form of the marginal distributions of the production subsystem and the inventory-replenishment subsystem. For the marginal distribution of the production subsystem we derive an explicit solution and for the marginal distribution of the inventory-replenishment subsystem we derive  an explicit solution resp.~a recursive algorithm for some special cases.
	
	% Add footnote without footnote mark
	\begingroup
	\renewcommand\thefootnote{}%
	\footnote{
		ORCID and email address:\\
		Sonja Otten \includegraphics[width=0.8em]{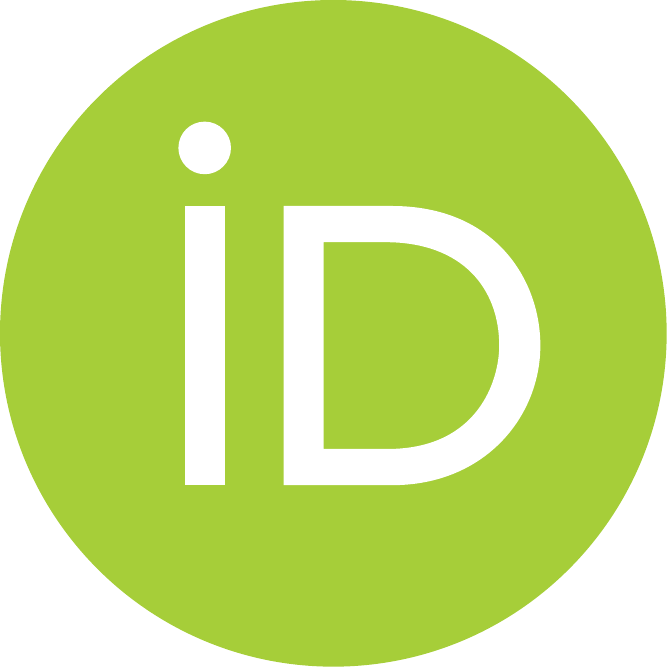}\hspace{0.4em}\url{https://orcid.org/0000-0002-3124-832X},
		sonja.otten@tuhh.de,\\

	}%

	\addtocounter{footnote}{-1}%
	\endgroup
	
\end{abstract}
\emph{MSC 2010 Subject Classification:} 60K25; 68M20; 90B22; 90B05; 90B06

\noindent \emph{Keywords:} queueing networks, inventory control, base stock policy, product-form steady state

%--------------------------------------------

\section{Introduction}
We consider a supply chain consisting of 
production-inventory systems at several locations which
are coupled by a common supplier. Demand of customers arrives at each
production system and is lost if the local inventory is depleted.
To satisfy a customer's demand a server at the production system takes
exactly one unit of raw material from the associated local inventory.
The supplier manufactures raw material to replenish the local inventories,
which are controlled by a continuous review base stock policy.
We focus on the research of the network's behaviour, where the supplier
consists only of one machine (single server) and replenishes the inventories
at all locations. The items of raw material are indistinguishable
(exchangeable). The routing of items depends on
the on-hand inventory at the locations with the aim to obtain `load balancing'. 
More precisely, we consider strict priorities, i.e.~the finished item of raw material is sent to the
location(s) with the highest difference between the on-hand inventory and the capacity
of the inventory.

Although we describe our system in terms of production and manufacturing, there are other applications where our
model can be used, e.g. distributed retail systems where customers' demand has to be satisfied from the local inventories and
delivering the goods to the customers' needs a non-negligible amount of time; the replenishment for the local retail stations
is provided by a production network. Another setting is a distributed set of repair stations where spare parts are needed to
repair the brought-in items which are held in local inventories. Production of the needed spare parts and sending them to the
repair stations is again due to a production network.

We use queueing theory and inventory theory to analyse the above-described production-inventory-replenishment system. Queueing
theory and inventory theory are fields of Operations Research with different methodologies to optimise e.g.~production
processes and inventory control. In classical Operations Research, queueing theory and inventory theory are often considered
as disjoint areas of research. On the other side, the emergence of complex supply chains calls for integrated queueing-inventory
models, where, as usual, the production systems are modelled by queueing systems and the inventories are modelled by classical inventories 
with replenishment policies. These integrated queueing-inventory models are the focus of our present
research. In particular, we develop a Markovian stochastic model of the production-inventory-replenishment system. 

We are able to prove that the stationary distribution has a product form. 
This means that the steady states of the production network and
the inventory-replenishment complex decouple asymptotically and the equilibrium for the
production subsystem decomposes in true independent coordinates. 
This product form structure of the joint stationary distribution is often characterised
as the global process being `separable'. Separability is an important (but rather rare) property
of complex systems.

The paper is organised as follows. In \prettyref{sec:literature}, we describe the related literature. 
In \prettyref{sec:BSN-STAR2-modeldescription}, we introduce our integrated
model for production and inventory management. To model our production-inventory-replenishment system as a queueing network, we
make a number of simplifying assumptions. 
This enables us to compute the stationary distribution and to show in \prettyref{sec:stationary-behaviour}
that it is of product form.  
In \prettyref{sec:strucural properties} we show some structural properties 
of the stationary inventory-replenishment subsystem.

\section{Related literature and own contributions \label{sec:literature}}
Relevant for our research are queueing theory and inventory control,
in particular integrated queueing-inventory models.\\

\textit{Literature on queueing theory} is overwhelming, so we point
only to the most relevant sources for our present investigation. Our
production systems are classical $M/M/1/\infty$ queueing systems
which constitute a network of parallel queues connected to the central
supplier queue, cf.\ Kelly~\cite{kelly:79} and Chao, Miyazawa,
and Pinedo~\cite{chao;miyazawa;pinedo:99} for general networks of
queues.

Special queueing networks, which model multi-station maintenance and
repair systems, are investigated by Ravid, Boxma, and Perry~\cite{ravid;boxma;perry:13}
and Daduna~\cite{daduna:90} and references therein. In these systems,
circulating items are ``exchangeable''. This feature will occur
in our model as well.\\ 

\textit{Literature on inventory theory} is, similar to that on queueing
theory, overwhelming, so we only point to some references closely
related to our investigations. There are two extreme
cases of arriving customers' reactions in the situation that inventory
is depleted when demand arrives (cf.\ Silver, Pyke, and Peterson~\cite{silver;pyke;peterson:98}):
Either backordering, which means that customers are willing to wait
for their demands to be fulfilled, or lost sales, which means that
demand is lost when no inventory is available on hand.

In classical inventory theory it is common to assume that excess demand
is backordered (Silver, Pyke, and Peterson~\cite{silver;pyke;peterson:98},
Zipkin~\cite[p. 40]{zipkin:2000}, Axsäter~\cite{axsaeter:2000}).
However, studies by Gruen, Corsten, and Bharadwaj~\cite{gruen;corsten;bharadwaj:2002}~and
Verhoef and Sloot~\cite{verhoef;sloot:2006} analyse customers' behaviour
in practice and show that in many retail settings most of the original
demand can be considered to be lost in case of a stockout.

For an overview of the literature on systems with lost sales we refer
to Bijvank and Vis~\cite{bijvank;vis:11}. They present a classification
scheme for the replenishment policies most often applied in literature
and practice, and they review the proposed replenishment policies,
including the base stock policy. According to van Donselaar and Broekmeulen~\cite{donselaar;broekmeulen:13}
``Their literature review confirms that there are only a limited
number of papers dealing with lost sales systems and the vast majority
of these papers make simplifying assumptions to make them analytically
tractable.''

Rubio and Wein~\cite{rubio;wein:96} and Zazanis~\cite{zazanis:94}
investigated classical single item and multi-item inventory systems.
Similar to our approach they used methods and models from queueing
theory to evaluate the performance of base stock control policies
in complex situations.

Reed and Zhang~\cite{reed;zhang:17} study a single item inventory
system under a base stock policy with backordering and a supplier
who consists of a multi-server production system. Their aim is to
minimise a combination of capacity, inventory and backordering costs.
They develop a square-root rule for the joint decision. Furthermore,
they justify the rule analytically in a many-server queue asymptotic
framework.

Because we consider queueing-inventory systems where inventories are
controlled by base stock policies, we mention here that Tempelmeier~\cite[p.\ 84]{tempelmeier:2005}
argued that base stock control is economically reasonable if the order
quantity is limited because of technical reasons. 

The base stock policy is ``(...)\ more suitable for item with low
demand, including the case of most spare parts'' \cite[p. 661]{rego:11}.

Morse~\cite[p.\ 139]{morse:1958} investigated (pure) inventory systems
that operate under a base stock policy. He gives a very simple example
where the concept ``re-order for each item sold'' is useful: Items
in inventory are bulky, and expensive (automobiles or TV sets\footnote{The paper is from 1958.}).
He uses queueing theory to model the inventory systems, analogously
to~\cite{reed;zhang:17}, etc.\\

\textit{Literature on integrated queueing-inventory models} (i.e.\ queueing
theory in combination with inventory theory) is also overwhelming. For a recent review we refer to Krishnamoorthy, Shajin and Narayanan~\cite{krishnamoorthy;shajin;narayanan:21}.

%---------------------------------------------------------------------------------------Load balancing-------------------------------------------
In this paper, we analyse an extension of the complex supply chain from \cite{otten;krenzler;daduna:15}, 
where routing of items depends on the on-hand inventory at the locations with the aim to obtain "load balancing". It can be differed between two load balancing policies.
In \cite[Section 3.4]{otten:17} we consider weak priorities, i.e.~the finished item of raw material is sent with greater probability to the location with higher difference
between the on-hand inventory and the capacity of the inventory.
In this paper, we focus on strict priorities, i.e.~the finished item of raw material is sent to the
location(s) with the highest difference between the on-hand inventory and the capacity
of the inventory. Therefore, we mention here some
\textit{Literature about load balancing policies with strict priorities}:\\
The research of such systems is motivated by state-dependent routing/branching
of customers. The ``(...)\ purpose of introducing flexible state-dependent
routing strategies is to optimally utilize network resources and to
minimise network delay and response times''~\cite[p.\ 1]{daduna:87}
and ``(...)\ that introducing state-dependent routing into product
form networks usually destroys the product-form of the steady state
probabilities''~\cite[p.\ 1]{daduna:87}. Product form solutions
under state-dependent routing are, for example, found in~\cite{pittel:79},~\cite{hordijk;dijk:84},~\cite{schassberger:84a}.

Several other stock allocation policies can, for example, be found
in the article of Abouee-Mehrizi and his coauthors~\cite{abouee-mehrizi;baron;berman:14}.
They consider a two-echelon inventory system with a capacitated centralized
production facility and several distribution centres. We will not
go into any greater detail in this allocation policies.

Daduna~\cite[p.\ 624]{daduna:85} and Towsley~\cite[pp.\ 327f.]{towsley:80}
argued that an optimal routing/branching policy for systems with identical
peripheral processors is by intuitive reasoning: ``Customers enter
the peripheral processor with the shortest queue''. This is equivalent
to our strict priorities for load balancing policy.

Chow and Kohler analyse the performance of two-processor distributed
computer systems under several dynamic load balancing policies in~\cite{chow;kohler:77}
and~\cite{chow;kohler:79}. They compare the performance and their
results indicate that a simple load balancing policy can significantly
improve the performance (turnaround time) of the system.
In~\cite{chow;kohler:77} they analyse the performance of homogeneous
(i.e.\ identical) two-processor distributed computer systems under
several dynamic load balancing policies. Their analysis is based on
the recursive solution technique and they illustrated the algorithm
for a sample system~\cite[Appendix, pp.\ 51f.]{chow;kohler:77}.
Their strategy ``join the shorter queue without channel transfer''
in Chow and Kohler's Model B~\cite[pp.\ 42f.]{chow;kohler:77} is
equivalent to our strict load balancing policy. 

Flatto and McKean study also Chow and Kohler's Model B in~\cite{flatto;mckean:77}
and derive by the generating function approach a complicated closed
form solution (cf.~\cite[Section 3, p. 261]{flatto;mckean:77}),
where the inter-arrival times of customers are exponentially distributed
with rate $1$.

Chow and Kohler present in~\cite{chow;kohler:79} a generalization
of the recursive solution technique. They apply the method to non-homogeneous
($=$ heterogeneous) two-processor systems with special properties
in~\cite[Section IV, pp.\ 358f.]{chow;kohler:79} and present a sample
system using the algorithm in~\cite[Appendix, pp.\ 360f.]{chow;kohler:79}.
They mention that the generalization of recursive solution ``(...)\ technique
for three or more processors does not appear to be straightforward''~\cite[p.\ 359]{chow;kohler:79}.\\

\noindent\textbf{Our main contributions} are the following:\\
We analyse an extension of the complex supply chain from \cite{otten;krenzler;daduna:15},
where routing of items depends on the on-hand inventory at the locations (with the aim to obtain "load balancing").
For the system with strict priorities for load balancing policy we develop
a Markov process. We prove that the stationary distribution has a
product form of the marginal distributions of the production subsystem
and of the inventory-replenishment subsystem.
We derive an explicit solution for the marginal distribution of the
production subsystem. Furthermore, 
for the special case with base stock levels equal to one we have derived
an explicit solution for the marginal distribution of the inventory-replenishment subsystem.  
For systems with base stock levels greater than one the marginal
distribution of the inventory-replenishment subsystem with two locations
can be obtained by a recursive method which is described by an algorithm.
On the other hand, our work is an extension of the investigations of Chow and Kohler~\cite{chow;kohler:77,chow;kohler:79}:
Their study is limited to two processors ($=$ our inventories without
production systems). 
For the heterogeneous case, our load balancing policy is slightly
different from that of Chow and Kohler~\cite[Section IV, pp.\ 358f.]{chow;kohler:79}
(for more details see~\prettyref{app:BSN-STAR2-non-identical}). 
Therefore, we construct a new algorithm.\\

\noindent\textbf{Notations and Conventions:}
\begin{itemize}
	\item $\mathbb{N}:=\left\{ 1,2,3,\ldots\right\} $, $\mathbb{N}_{0}:=\{0\}\cup\mathbb{N}$.
	
	\item The vector $\mathbf{0}$ is a row vector of appropriate size with
	all entries equal to $0$. The vector $\mathbf{e}$ is a column vector
	of appropriate size with all entries equal to $1$. The vector $\mathbf{e}_{i}=(0,\ldots,0,\underbrace{1}_{\mathclap{i-\text{th element}}},0,\ldots,0)$
	is a vector of appropriate dimension.
	
	\item The notation $\subset$ between sets means ``subset or equal'' and
	$\subsetneq$ means ``proper subset''.  For
	a set $A$ we denote by $\vert A\vert$ the number of elements in
	$A$.
	
	\item $1_{\left\{ expression\right\} }$ is the indicator function which
	is $1$ if $expression$ is true and $0$ otherwise.
	
	\item Empty sums are 0, and empty products are 1.
	
	\item Throughout this article it is assumed that all random variables are
	defined on a common probability space $(\Omega,{\cal F},P)$. Furthermore,
	by Markov process we mean time-homogeneous continuous-time strong
	Markov process with discrete state space ($=$ Markov jump process).
	Without further mentioning all Markov processes are assumed to be
	regular and have cadlag paths, i.e.\ each path of a process is right-continuous
	and has left limits everywhere. We call a Markov process regular if
	it is non-explosive (i.e.\ the sequence of jump times of the process
	diverges almost surely), its transition intensity matrix is conservative
	(i.e.\ row sums are $0$) and stable (i.e.\ all diagonal elements
	of the transition intensity matrix are finite).
\end{itemize}

\section{Description of the general model\label{sec:BSN-STAR2-modeldescription}}

The supply chain of interest is depicted in \prettyref{fig:BSN-STAR2-figure-base-stock-model}.
We have a set of locations $\Jset:=\left\{ 1,2,\dots,J\right\} $,
$J>1$.\index{JJ@$\Jset$, set of locations}\index{J@$J$, number of locations}
Each of the locations consists of a production system with an attached
inventory. The inventories are replenished by a single central supplier,
which is referred to as workstation $J+1$ and manufactures raw material
for all locations. The items of raw material are indistinguishable
(exchangeable).\\
\begin{figure}[h]
	\centering{}\includegraphics[width=1\textwidth]{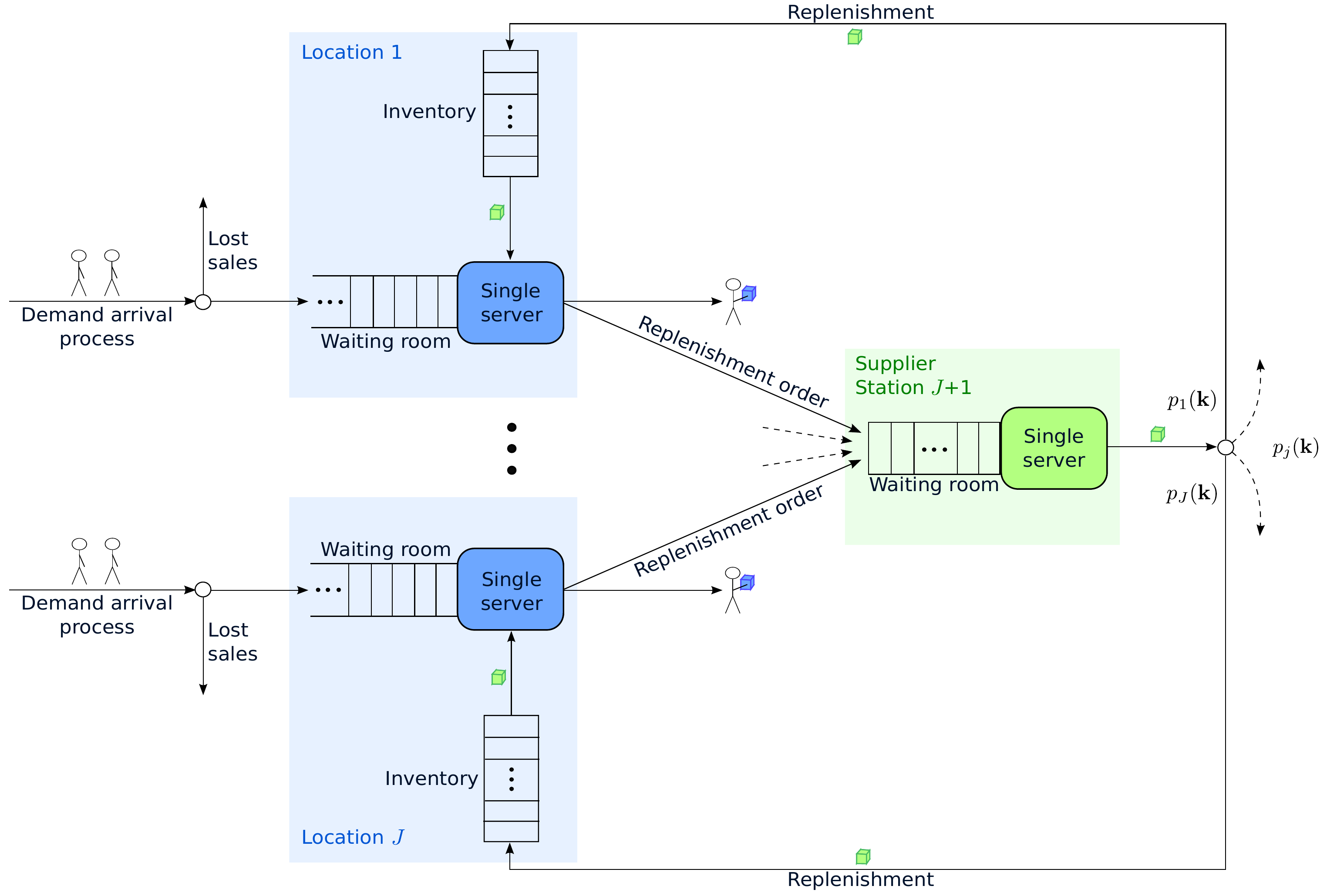}
	\caption{\label{fig:BSN-STAR2-figure-base-stock-model}Supply chain with base
		stock policies and load balancing policies}
\end{figure}

\textbf{Facilities in the supply chain.} Each production system $j\in\Jset$
consists of a single server (machine) with infinite waiting room that
serves customers on a make-to-order basis under a FCFS regime. Customers
arrive one by one at the production system $j$ according to a Poisson
process with rate $\lambda_{j}>0$\index{$\lambda$j@$\lambda_{j}$, arrival rate at location $j$}
and require service. To satisfy a customer's demand the production
system needs exactly one item of raw material, which is taken from
the associated local inventory. When a new customer arrives at a location
while the previous customers' order is not finished, this customer
will wait. If the inventory is depleted at location $j$, the customers
who are already waiting in line will wait, but new arriving customers
at this location will decide not to join the queue and are lost (``local
lost sales'').

The service requests at the locations are exponentially-$1$ distributed.
All service requests constitute an independent family of random variables
which are independent of the arrival streams. The service at location
$j\in\Jset$ is provided with local queue-length-dependent intensity.
If there are $n_{j}>0$\index{nj@$n_{j}$, number of customers present at location $j$}
customers present at location $j$, either waiting or in service (if
any), and if the inventory is not depleted, the service intensity
is $\mu_{j}(n_{j})>0$.\index{$\mu$j@$\mu_{j}(n_{j})$, service intensity if there are $n_{j}$
	customers present at location $j$} If the server is ready to serve a customer who is at the head of
the line, and the inventory is not depleted, the service immediately starts. Otherwise, the
service starts at the instant of time when the next replenishment
arrives at the local inventory.

The inventory at location $j\in\Jset$ is controlled by prescribing
a local base stock level $b_{j}\geq1$,\index{bj@$b_{j}$, base stock level at location $j$}
which is the maximal size of the inventory there, we denote $\textbf{\ensuremath{\bvect}}:=\left(b_{j}:j\in\Jset\right)$.

The central supplier (which is referred to as workstation $J+1$)
consists of a single server (machine) and a waiting room under FCFS
regime. At most $\sum_{j\in\Jset}b_{j}-1$ replenishment orders are
waiting at the central supplier. Service times at the central supplier
are exponentially distributed with parameter $\nu>0$.\index{$\nu$@$\nu$, service rate of the central supplier}\\

\textbf{Routing in the supply chain. }A served customer departs from
the system immediately after service and the associated consumed raw
material is removed from the inventory and an order of one item is
placed at the central supplier at this time instant (``base stock
policy''). 

A finished item of raw material departs from the central supplier
immediately and is sent with probability $p_{j}(\kvect)$\index{pjk@$p_{j}(\kvect)$, state-dependent routing probability to location
	$j$}, independent of the network's history, given $\kvect$, to location
$j$, $j\in\Jset$, if the state of the inventory-replenishment subsystem
is $\kvect$.
We consider the following \textit{load	balancing policy} with strict priorities:
The finished item of raw material is sent to location $j\in\Jset$
with probability
\[
p_{j}(\kvect)=\begin{cases}
1, & \ \text{if }\{j\}=\underset{i\in\Jset}{\arg\max}(b_{i}-k_{i}),\\
\frac{1}{\vert\underset{i\in\Jset}{\arg\max}(b_{i}-k_{i})\vert}<1, & \ \text{if }\{j\}\subsetneq\underset{i\in\Jset}{\arg\max}(b_{i}-k_{i}),\\
0, & \ \text{if }j\notin\underset{i\in\Jset}{\arg\max}(b_{i}-k_{i}),
\end{cases}
\]
i.e.\  to the location(s) with the highest difference between the
on-hand inventory and the capacity of the inventory ($=$ base stock
level), if the inventory is not full at this/these location(s) (this
means that the on-hand inventory level at this/these location(s) is
lower than the base stock level). The routing probabilities out of
the central supplier must sum to one if there is at least one order
at the central supplier.

It is assumed that transmission times for orders are negligible and
set to zero and that transportation times between the central supplier
and the local inventories are negligible.
The usual independence assumptions are assumed to hold as well.\\

\textbf{To obtain a Markovian process} description of the integrated
queueing-inventory system, we denote by $X_{j}(t)$\index{Xj@$X_{j}(t)$, number of customers present at location $j$ at time
	$t\geq0$} the number of customers present at location $j\in\Jset$ at time
$t\geq0$\index{t@$t$, time}, either waiting or in service (queue
length). By $Y_{j}(t)$\index{Yj@$Y_{j}(t)$, size of the inventory at location $j$ at time $t\geq0$}
we denote the size of the inventory at location $j\in\Jset$ at time
$t\geq0$. By $W_{J+1}(t)$\index{Wj@$W_{J+1}(t)$, number of replenishment orders at the central supplier
	at time $t\geq0$} we denote the number of replenishment orders at the central supplier
at time $t\geq0$, either waiting or in service (queue length). 

We define the joint queueing-inventory process of this system by\index{ZZ@$Z$, joint queueing-inventory process}
\[
Z=\left(\left(X_{1}(t),\ldots,X_{J}(t),Y_{1}(t),\ldots,Y_{J}(t),W_{J+1}(t)\right)\,:\,t\geq0\right).
\]
Then, due to the usual independence and memoryless assumptions $Z$
is a homogeneous Markov process, which we assume to be irreducible
and regular. The state space of $Z$ is
\[E=\left\{ \left(\nvect,\kvect\right):\nvect\in\mathbb{N}_{0}^{J},\,\kvect\in K\right\}\]
with
\begin{align*}
K&:=\Big\{(k_{1},\dots,k_{J},k_{J+1})\vert 0\leq k_{j}\leq b_{j},\:j=1,\dots,J,\:k_{J+1}=\sum_{j=1}^{J}(b_{j}-k_{j})\Big\}\subset\mathbb{N}_{0}^{J+1}.
\end{align*}
Note the redundancy in the state space: $W_{J+1}(t)=\sum_{j\in\Jset}b_{j}-\sum_{j\in\Jset}Y_{j}(t)$.
We prefer to carry all information explicitly with because the dynamics
of the system are easier visible.\\

\noindent\textbf{Discussion of the modelling assumptions}\\
We have imposed several simplifying assumptions on the production-inventory-replenishment
system to obtain explicit and simple-to-calculate performance metrics
of the system, which give insights into its long-time and stationary
behaviour. This enables a parametric and sensitivity analysis that
is easy to perform.

First, the assumption of exponentially distributed inter-arrival and
service times are standard in the literature and are the best first-order
approximations. The locally state-dependent service rates are also
common and give quite a bit of flexibility. The lead time is composed
of the waiting time plus the production time at the central supplier.
Therefore, it is more complex than exponential, constant or even zero
lead times (which are often assumed in standard inventory literature).
Zero lead times in our systems would result in almost trivial extensions
of the queueing systems.

Second, we assume that the local base stock levels are  positive (i.e.\ $b_{j}\geq1$
at location $j$). This assumption can be made without loss of generality.
Otherwise, all customers at location $j$ would be lost, which is
the same as excluding location $j$ from the production-inventory-replenishment
system.

Third, the assumption of zero transportation times can be removed
by inserting special (virtual) $M/G/\infty$ workstations into the
network.

\section{Limiting and stationary behaviour \label{sec:stationary-behaviour}}

The queueing-inventory process $Z$ has an infinitesimal generator
$\mathbf{Q}=\left(q(z;\tilde{z}):z,\tilde{z}\in E\right)$ with the
following  transition rates for $(\nvect,\kvect)\in E$:\index{QzQ0@$\mathbf{Q},$ infinitesimal generator}\index{z@$z$, state of the process $Z$}\index{q@$q(\cdot,\cdot),$ transition rates}\label{BSN-STAR-transitionrates}
\begin{align*}
q\left((\nvect,\kvect);(\nvect+\evect_{i},\kvect)\right) & =\lambda_{i}\cdot1_{\left\{ k_{i}>0\right\} }, &  & i\in\Jset,\\
q\left((\nvect,\kvect);(\nvect-\evect_{i},\kvect-\evect_{i}+\evect_{J+1})\right) & =\mu_{i}(n_{i})\cdot1_{\left\{ n_{i}>0\right\} }\cdot1_{\left\{ k_{i}>0\right\} }, &  & i\in\Jset,\\
q\left((\nvect,\kvect);(\nvect,\kvect+\evect_{i}-\evect_{J+1})\right) & =\supplierrate\cdot\routingprob\cdot1_{\left\{ k_{i}<b_{i}\right\} }, &  & i\in\Jset.
\end{align*}
Note that $k_{J+1}>0$ holds if $k_{i}<b_{i}$
for some $i\in\Jset$.

Furthermore, $q(z;\tilde{z})=0$ for any other pair $z\neq\tilde{z}$,
and
\[
q\left(z;z\right)=-\sum_{\substack{\tilde{z}\in E,\\
		z\neq\tilde{z}
	}
}q\left(z;\tilde{z}\right)\qquad\forall z\in E.
\]

\begin{prop}
	\label{prop:BSN-STAR2-limiting-distribution-x}There exists a strictly
	positive measure $\widetilde{\theta}=(\widetilde{\theta}(\kvect):\kvect\in K)$,
	which will be provided below, such that the measure $\mathbf{x}:=\left(x\left(\nvect,\kvect\right):\left(\nvect,\kvect\right)\in E\right)$
	with\index{$\xi$1@$\widetilde{\xi}$, stationary measure for the queues at the
		locations (production systems)}\index{$\xi$j1@$\widetilde{\xi}_{j}(n_{j})$, stationary measure for the
		queue at location (production system) $j$}\index{$\theta$@$\widetilde{\theta}$, stationary measure for the inventory-replenishment
		subsystem}
	\begin{equation}
	x\left(\nvect,\kvect\right)=\widetilde{\xi}(\nvect)\cdot\widetilde{\theta}\left(\kvect\right),\label{eq:BSN-STAR2-stationary-distribution-x}
	\end{equation}
	where
	\begin{equation}
	\widetilde{\xi}(\nvect)=\prod_{j\in\Jset}\widetilde{\xi}_{j}(n_{j}),\qquad\qquad\widetilde{\xi}_{j}(n_{j})=\prod_{\ell=1}^{n_{j}}\frac{\lambda_{j}}{\mu_{j}(\ell)},\quad n_{j}\in\mathbb{N}_{0},\ j\in\Jset,\label{eq:BS-stationary-distribution-pi-load}
	\end{equation}
	solves the global balance equations $\mathbf{x}\cdot\mathbf{Q=0}$
	and is therefore stationary for $Z$. Consequently, $\mathbf{x}$
	is strictly positive.
\end{prop}

\begin{proof}
	\textbf{The global balance equations} $\mathbf{x}\cdot\mathbf{Q=0}$
	of the stochastic queueing-inventory process $Z$ are given for $\left(\nvect,\kvect\right)\in E$
	by
	\begin{align*}
	& \phantomeq x\left(\nvect,\kvect\right)\cdot\Big(\sum_{i\in\Jset}\lambda_{i}\cdot1_{\left\{ k_{i}>0\right\} }+\sum_{i\in\Jset}\mu_{i}(n_{i})\cdot1_{\left\{ n_{i}>0\right\} }\cdot1_{\left\{ k_{i}>0\right\} }+\sum_{i\in\Jset}\supplierrate\cdot p_{i}(\kvect)\cdot1_{\left\{ k_{i}<b_{i}\right\} }\Big)\\
	& =\sum_{i\in\Jset}x\left(\nvect-\evect_{i},\kvect\right)\cdot\lambda_{i}\cdot1_{\left\{ n_{i}>0\right\} }\cdot1_{\left\{ k_{i}>0\right\} }\\
	& \phantomeq+\sum_{i\in\Jset}x\left(\nvect+\evect_{i},\kvect+\evect_{i}-\evect_{J+1}\right)\cdot\mu_{i}(n_{i}+1)\cdot1_{\left\{ k_{i}<b_{i}\right\} }\\
	& \phantomeq+\sum_{i\in\Jset}x\left(\nvect,\kvect-\evect_{i}+\evect_{J+1}\right)\cdot\supplierrate\cdot p_{i}(\kvect-\evect_{i}+\evect_{J+1})\cdot1_{\left\{ k_{i}>0\right\} }.
	\end{align*}
	It has to be shown that the stationary measure from \prettyref{prop:BSN-STAR2-limiting-distribution-x}
	satisfies these global balance equations. 
	%----------------------neu
	Substitution of~\prettyref{eq:BSN-STAR2-stationary-distribution-x} and~\prettyref{eq:BS-stationary-distribution-pi-load}
	into the global balance equations directly leads to
	\begin{align*}
	& \phantomeq\left(\prod_{j\in\Jset}\widetilde{\xi}_{j}(n_{j})\right)\cdot\widetilde{\theta}\left(\kvect\right)\\
	& \phantomeq\cdot\Big(\sum_{i\in\Jset}\lambda_{i}\cdot1_{\left\{ k_{i}>0\right\} }+\sum_{i\in\Jset}\mu_{i}(n_{i})\cdot1_{\left\{ n_{i}>0\right\} }\cdot1_{\left\{ k_{i}>0\right\} }+\sum_{i\in\Jset}\supplierrate\cdot p_{i}\cdot1_{\left\{ k_{i}<b_{i}\right\} }\Big)\\
	& =\sum_{i\in\Jset}\left(\prod_{{j\in\Jset\backslash\left\{ i\right\} }}\widetilde{\xi}_{j}(n_{j})\right){\widetilde{\xi}_{i}(n_{i}-1)}\cdot\widetilde{\theta}\left(\kvect\right)\cdot{\lambda_{i}}\cdot1_{\left\{ n_{i}>0\right\} }\cdot1_{\left\{ k_{i}>0\right\} }\\
	& \phantomeq+\sum_{i\in\Jset}\left(\prod_{j\in\Jset\backslash\left\{ i\right\} }\widetilde{\xi}_{j}(n_{j})\right)\cdot\widetilde{\xi}_{i}(n_{i}+1)\cdot\widetilde{\theta}\left(\kvect+\evect_{i}-\evect_{J+1}\right)\cdot\mu_{i}(n_{i}+1)\cdot1_{\left\{ k_{i}<b_{i}\right\} }\\
	& \phantomeq+\sum_{i\in\Jset}\left(\prod_{j\in\Jset}\widetilde{\xi}_{j}(n_{j})\right)\cdot\widetilde{\theta}\left(\kvect-\evect_{i}+\evect_{J+1}\right)\cdot\supplierrate\cdot p_{i}\cdot1_{\left\{ k_{i}>0\right\} }.
	\end{align*}
	By substitution of~\prettyref{eq:BS-stationary-distribution-pi-load}
	we obtain
	\begin{align*}
	& \phantomeq\left(\prod_{j\in\Jset}\widetilde{\xi}_{j}(n_{j})\right)\cdot\widetilde{\theta}\left(\kvect\right)\\
	& \phantomeq\cdot\Big(\sum_{i\in\Jset}\lambda_{i}\cdot1_{\left\{ k_{i}>0\right\} }+\sum_{i\in\Jset}\mu_{i}(n_{i})\cdot1_{\left\{ n_{i}>0\right\} }\cdot1_{\left\{ k_{i}>0\right\} }+\sum_{i\in\Jset}\supplierrate\cdot p_{i}\cdot1_{\left\{ k_{i}<b_{i}\right\} }\Big)\\
	& =\sum_{i\in\Jset}\left(\prod_{j\in\Jset}\widetilde{\xi}_{j}(n_{j})\right)\cdot\widetilde{\theta}\left(\kvect\right)\cdot\mu_{i}(n_{i})\cdot1_{\left\{ n_{i}>0\right\} }\cdot1_{\left\{ k_{i}>0\right\} }\\
	& \phantomeq+\sum_{i\in\Jset}\left(\prod_{j\in\Jset}\widetilde{\xi}_{j}(n_{j})\right)\cdot\widetilde{\theta}\left(\kvect+\evect_{i}-\evect_{J+1}\right)\cdot\lambda_{i}\cdot1_{\left\{ k_{i}<b_{i}\right\} }\\
	& \phantomeq+\sum_{i\in\Jset}\left(\prod_{j\in\Jset}\widetilde{\xi}_{j}(n_{j})\right)\cdot\widetilde{\theta}\left(\kvect-\evect_{i}+\evect_{J+1}\right)\cdot\supplierrate\cdot p_{i}\cdot1_{\left\{ k_{i}>0\right\} }.
	\end{align*}
	Cancelling $\left(\prod_{j\in\Jset}\widetilde{\xi}_{j}(n_{j})\right)$
	and the sums with the terms $\mu_{i}(n_{i})\cdot1_{\left\{ n_{i}>0\right\} }\cdot1_{\left\{ k_{i}>0\right\} }$
	on both sides of the equation leads to
	\begin{align}
	& \phantomeq\widetilde{\theta}\left(\kvect\right)\Big(\sum_{i\in\Jset}\lambda_{i}\cdot1_{\left\{ k_{i}>0\right\} }+\sum_{i\in\Jset}\supplierrate\cdot p_{i}(\kvect)\cdot1_{\left\{ k_{i}<b_{i}\right\} }\Big)\nonumber \\
	& =\sum_{i\in\Jset}\widetilde{\theta}\left(\kvect+\evect_{i}-\evect_{J+1}\right)\cdot\lambda_{i}\cdot1_{\left\{ k_{i}<b_{i}\right\} }\nonumber \\
	& \phantomeq+\sum_{i\in\Jset}\widetilde{\theta}\left(\kvect-\evect_{i}+\evect_{J+1}\right)\cdot\supplierrate\cdot p_{i}(\kvect-\evect_{i}+\evect_{J+1})\cdot1_{\left\{ k_{i}>0\right\} }.\qquad\label{eq:BSN-STAR2-theta-equation}
	\end{align}
	
	An inspection of the system~\prettyref{eq:BSN-STAR2-theta-equation}
	reveals that it is a ``generator equation'', i.e.\ the global balance
	equation $\widetilde{\theta}\cdot\mathbf{Q}_{red}=0$ for a suitably
	defined ergodic Markov process on state space $K$ with ``reduced
	generator'' $\mathbf{Q}_{red}=\left(q_{red}(\kvect;\widetilde{\kvect}):\kvect,\widetilde{\kvect}\in K\right)$\index{QzQred@$\mathbf{Q}_{red}$, reduced generator}
	with the following  transition rates for $\kvect\in K$:
	\begin{align*}
	q_{red}\left(\kvect;\kvect-\evect_{i}+\evect_{J+1}\right) & =\lambda_{i}\cdot1_{\left\{ k_{i}>0\right\} }, &  & i\in\Jset,\\
	q_{red}\left(\kvect;\kvect+\evect_{i}-\evect_{J+1}\right) & =\supplierrate\cdot\routingprob\cdot1_{\left\{ k_{i}<b_{i}\right\} }, &  & i\in\Jset.
	\end{align*}
	The Markov process generated by $\mathbf{Q}_{red}$ is irreducible
	on $K$ and therefore \prettyref{eq:BSN-STAR2-theta-equation} has
	a solution which is unique up to a multiplicative constant, which
	yields $\widetilde{\theta}$.
\end{proof}
\begin{rem}
	In \prettyref{subsec:BS-State-homo-hetero-theta}, the marginal measure
	$\widetilde{\theta}$ is derived in explicit form for the special
	case with base stock levels $b_{j}=1$, $j\in\Jset$.
	For systems with two locations and base stock levels greater than
	one ($b_{j}>1$, $j\in\Jset$), the marginal measure $\widetilde{\theta}$
	can be obtained by a recursive method which is described by the algorithm
	given in \prettyref{app:BSN-STAR2-iterative-algorithm}.
\end{rem}

Recall that the system is irreducible and regular. Therefore, if $Z$
has a stationary and limiting distribution, this is uniquely defined.
\begin{definition}
	For the queueing-inventory process $Z$ on state space $E$, whose
	limiting distribution exists, we define
	\[
	\pi:=\left(\pi\left(\nvect,\kvect\right):\left(\nvect,\kvect\right)\in E\right),\quad\pi\left(\nvect,\kvect\right):=\lim_{t\rightarrow\infty}P\left(Z(t)=(\nvect,\kvect)\right)
	\]
	and the appropriate marginal distributions
	\[
	\xi:=\left(\xi\left(\nvect\right):\nvect\in\mathbb{N}_{0}^{J}\right),\quad\xi\left(\nvect\right):=\lim_{t\rightarrow\infty}P\left(\left(X_{1}(t),\ldots,X_{J}(t)\right)=\nvect\right),
	\]
	\[
	\theta:=\left(\theta\left(\kvect\right):\kvect\in K\right),\quad\theta\left(\kvect\right):=\lim_{t\rightarrow\infty}P\left(\left(Y_{1}(t),\ldots,Y_{J}(t),W_{J+1}(t)\right)=\kvect\right).
	\]
\end{definition}

\begin{thm}\label{thm:BSN-star2-ergo}
	The queueing-inventory process $Z$ is
	ergodic if and only if for $j\in\Jset$
	\[
	\sum_{n_{j}\in\mathbb{N}_{0}}\prod_{\ell=1}^{n_{j}}\frac{\lambda_{j}}{\mu_{j}(\ell)}<\infty.
	\]
	If $Z$ is ergodic, then its unique limiting and stationary distribution
	is\index{$\pi$@$\pi$, limiting and stationary distribution of $Z$}
	\begin{equation}
	\pi\left(\nvect,\kvect\right)=\xi(\nvect)\cdot\theta\left(\kvect\right),\label{eq:BSN-STAR2-stationary-distribution}
	\end{equation}
	with\index{$\xi$@$\xi$, stationary marginal distributions of the queues at
		the locations (production systems)}\index{$\xi$j@$\xi_{j}(n_{j})$, stationary distribution of the queue at
		location (production system) $j$}\index{$\theta$@$\theta$, stationary marginal distribution of the inventory-replenishment
		subsystem}
	\begin{equation}
	\xi(\nvect)=\prod_{j\in\Jset}\xi_{j}(n_{j}),\qquad\qquad\xi_{j}(n_{j})=C_{j}^{-1}\prod_{\ell=1}^{n_{j}}\frac{\lambda_{j}}{\mu_{j}(\ell)},\quad n_{j}\in\mathbb{N}_{0},\ j\in\Jset,\label{eq:BSN-STAR2-stationary-pi}
	\end{equation}
	and normalisation constants\index{CzCj@$C_{j}$, normalization constant}
	\[
	C_{j}=\sum_{n_{j}\in\mathbb{N}_{0}}\prod_{\ell=1}^{n_{j}}\frac{\lambda_{j}}{\mu_{j}(\ell)}
	\]
	and $\theta$ is the probabilistic solution of \prettyref{eq:BSN-STAR2-theta-equation}.
\end{thm}

\begin{proof}
	$Z$ is ergodic if and only if the strictly positive measure $\mathbf{x}$
	of the global balance equation $\mathbf{x}\cdot\mathbf{Q}=\mathbf{0}$
	from \prettyref{prop:BSN-STAR2-limiting-distribution-x} can be normalised
	(i.e.\ $\sum_{\nvect\in\mathbb{N}_{0}}\sum_{\kvect\in K}x(\nvect,\kvect)<\infty$).
	Because of \prettyref{prop:BSN-STAR2-limiting-distribution-x} it
	holds
	\[
	\sum_{\nvect\in\mathbb{N}_{0}}\sum_{\kvect\in K}x(\nvect,\kvect)=\sum_{\nvect\in\mathbb{N}_{0}}\widetilde{\xi}(\nvect)\cdot\sum_{\kvect\in K}\widetilde{\theta}\left(\kvect\right)=\left(\prod_{j\in\Jset}\sum_{n_{j}\in\mathbb{N}_{0}}\prod_{\ell=1}^{n_{j}}\frac{\lambda_{j}}{\mu_{j}(\ell)}\right)\cdot\sum_{\kvect\in K}\widetilde{\theta}\left(\kvect\right).
	\]
	Hence, since $K$ is finite, the measure $\mathbf{x}$ from \prettyref{prop:BSN-STAR2-limiting-distribution-x}
	can be normalised if and only if $\sum_{n_{j}\in\mathbb{N}_{0}}\prod_{\ell=1}^{n_{j}}\frac{\lambda_{j}}{\mu_{j}(\ell)}<\infty$
	for all $j\in\Jset$.
	
	Consequently, if the process is ergodic, the limiting and stationary
	distribution $\pi$ is given by
	\[
	\pi(\nvect,\kvect)=\frac{x(\nvect,\kvect)}{\sum_{\nvect\in\mathbb{N}_{0}}\sum_{\kvect\in K}x(\nvect,\kvect)},
	\]
	where $x(\nvect,\kvect)$ is given in \prettyref{prop:BSN-STAR2-limiting-distribution-x}.
\end{proof}
\begin{rem}
	\label{rem:remark-pf}The expression \prettyref{eq:BSN-STAR2-stationary-distribution}
	shows that the two-component production-inventory-replenishment system
	is separable, the steady states of the production network and the
	inventory-replenishment complex decouple asymptotically.
	
	Representation~\prettyref{eq:BSN-STAR2-stationary-pi} shows that
	the equilibrium for the production subsystem decomposes in true independent
	coordinates. A product structure of the stationary distribution as
	\[
	\xi(\nvect)=\prod_{j\in\Jset}\xi_{j}(n_{j})=\prod_{j\in\Jset}C_{j}^{-1}\prod_{\ell=1}^{n_{j}}\frac{\lambda_{j}}{\mu_{j}(\ell)}
	\]
	is commonly found for standard Jackson networks \cite{jackson:57}
	and their relatives. In Jackson networks servers are ``non-idling'',
	i.e.\ they are always busy as long as customers are present at the
	respective node. In our production network, however, servers may be
	idle while there are customers waiting because a replenishment needs
	to arrive first. Consequently, the product form~\prettyref{eq:BSN-STAR2-stationary-distribution}
	has been unexpected to us.
	
	Our production-inventory-replenishment system can be considered as a ``Jackson
	network in a random environment'' in~\cite[Section 4]{krenzler;daduna;otten:16}.
	We can interpret the inventory-replenishment subsystem, which contributes
	via $\theta$ to \prettyref{thm:BSN-star2-ergo}, as a ``random environment''
	for the production network of nodes $\Jset$, which is a Jackson network
	of parallel servers.
	Taking into account the results of~\cite[Theorem 4.1]{krenzler;daduna;otten:16}
	we conclude from the hindsight that decoupling of the queueing process
	$\left(X_{1},\dots,X_{J}\right)$ and the process $\left(Y_{1},\dots,Y_{J},W_{J+1}\right)$,
	i.e.\ the formula~\prettyref{eq:BSN-STAR2-stationary-distribution},
	is a consequence of that Theorem 4.1. 
	
	Our direct proof of \prettyref{prop:BSN-STAR2-limiting-distribution-x}
	is much shorter than embedding the present model into the general
	framework of~\cite{krenzler;daduna;otten:16}.
\end{rem}

\noindent\textbf{Structural properties of the integrated system\label{sect:structureinf}}\\
The investigations in this section rely on the fact that the product
form of the stationary distribution ($=$ separability) makes structures
easily visible that are hard to detect by simulations or by direct
numerical investigations. As a byproduct we demonstrate the power
of product form calculus.

\noindent\textit{Ergodicity:\label{subsec:BS-structure-ergodicity}}
As shown in \prettyref{thm:BSN-star2-ergo}, ergodicity
is determined by $\sum_{n_{j}\in\mathbb{N}_{0}}\prod_{\ell=1}^{n_{j}}\frac{\lambda_{j}}{\mu_{j}(\ell)}<\infty$,
$j\in\Jset$, because $K$ is finite. Hence, ergodicity is determined
by the parameters of the isolated queueing system without the inventory
system at the locations. For instance, if ${\mu_{j}(\ell)}=\mu_{j},\forall\ell$,
then $\lambda_{j}<\mu_{j},\forall j$, is the correct condition for
stabilizing the entire system.

Noteworthy is that the extra idle times of the servers at the production
systems do not destroy ergodicity due to the necessary replenishments.
The reason behind this is that the local lost sales at the individual
servers at these locations automatically balance a possible bottleneck
behaviour of the central supplier.

\noindent\textit{Insensitivity and robustness:\label{subsec:BS-structure-insensitivity-and-robustness}}
Sensitivity analysis is an important topic in classical inventory
theory and is often hard to perform. In our model the stationary distribution
$\pi\left(\nvect,\kvect\right)=\left(\prod_{j\in\Jset}\xi_{j}(n_{j})\right)\cdot\theta\left(\kvect\right)$
reveals strong insensitivity properties of the system which make sensitivity
analysis amenable: The steady state behaviour of the subnetwork consisting
of inventories and the central supplier does not change when the service
rates at the locations are changed as long as the global system remains
ergodic. Therefore $\theta$ is robust against estimation errors in
determining the $\mu_{j}(\cdot)$. Vice versa, the distribution $\xi_{j}(n_{j})$
is robust against changes in the inventory-replenishment network as long
as the demand intensity and the service rates are maintained.

\begin{rem}
	A cost analysis can be performed as for the model without load balancing in \cite[pp.~11f.]{otten;krenzler;daduna:15}. Note a corrected version of the definition of the cost function can be found in \cite[Section 2.5, pp.~25f.]{otten:17}.
\end{rem}

\subsection{Calculation of $\widetilde{\theta}$}
\subsubsection{Special case: $b_{j}=1$,	$j\in\protect\Jset$\label{subsec:BS-State-homo-hetero-theta}}

In this section, we will solve the global balance equation $\widetilde{\theta}\cdot\mathbf{Q}_{red}=\mathbf{0}$
for the special case with base stock levels $b_{j}=1$, $j\in\Jset$. In this special case, our model with strict priorities is identical to the model with weak priorities in \cite[Section 3.4]{otten:17}, where the finished item of raw material is sent with greater probability to the location with higher difference between the on-hand inventory and the capacity of the inventory. For the sake of completeness we present the proof here.

We recall the notation for the inventory-replenishment subsystem
\[
\widetilde{\theta}(\overbrace{k_{1},k_{2},\ldots,k_{J}}^{\substack{\text{inventories}\\
		\text{at locations}
	}
},\overbrace{k_{J+1}}^{\text{supplier}}).
\]
Furthermore, in this special case it holds $p_{i}(\kvect)=\frac{1}{J-\sum_{j\in\Jset}k_j}$ if $k_i=0$, $i\in \Jset$.

\begin{prop}
	\label{prop:BS-STATE-homo-hetero-theta}The strictly positive measure
	$\widetilde{\theta}=\left(\widetilde{\theta}(\kvect):\kvect\in K\right)$
	of the inventory-replenishment subsystem with base stock levels $b_{1}=\cdots=b_{J}=1$
	is given by
	%-------------neu ---------------
	\begin{align}
	\widetilde{\theta}(\kvect)=\widetilde{\theta}(k_{1},k_{2},\ldots,k_{J},k_{J+1}) & 
	=\left(\prod_{\ell=0}^{(\sum_{j\in\Jset}k_{j})-1}\frac{1}{J-\ell}\right)
	\cdot \left(\prod_{j\in \Jset}\left(\frac{1}{\lambda_j}\right)^{k_{j}}\right)
	\cdot \left(\frac{1}{\nu}\right)^{k_{J+1}}
	\label{eq:BS-homo-j-bsp-00}
	\end{align}
\end{prop}

\begin{rem}
	As we have mentioned before  \prettyref{eq:BSN-STAR2-theta-equation} has a solution which is unique up to a multiplicative constant, which
	yields $\widetilde{\theta}$. In \cite[Proposition 3.3.6, p.~50]{otten:17} another stationary measure $\widetilde{\theta}_1$ is presented which differs from our stationary measure up to a multiplicative constant, i.e.~ $\widetilde{\theta}=\left(\tfrac{1}{\nu}\right)^{J}\cdot \widetilde{\theta}_1$.
\end{rem}
\begin{proof}
	It has to be shown that the stationary measure \prettyref{eq:BS-homo-j-bsp-00}
	satisfies the global balance equations $\widetilde{\theta}\cdot\mathbf{Q}_{red}=\mathbf{0}$,
	which are given for $\kvect\in K$ by
	\begin{align*}
	& \phantomeq\widetilde{\theta}\left(\kvect\right)\cdot\Big(\sum_{i\in\Jset}\lambda_{i}\cdot 1_{\left\{ k_{i}=1\right\} }
	+\sum_{i\in\Jset}\supplierrate\cdot p_{i}(\kvect)\cdot 1_{\left\{ k_{i}=0\right\} }\Big)\\
	& =\sum_{i\in\Jset}\widetilde{\theta}\left(\kvect+\evect_{i}-\evect_{J+1}\right)\cdot\lambda_{i}\cdot1_{\left\{ k_{i}=0\right\} }\\
	& \phantomeq+\sum_{i\in\Jset}\widetilde{\theta}\left(\kvect-\evect_{i}+\evect_{J+1}\right)\cdot\supplierrate\cdot p_{i}(\kvect-\evect_{i}+\evect_{J+1})\cdot1_{\left\{ k_{i}=1\right\} }
	\\
	\Leftrightarrow\quad &
	\phantomeq\widetilde{\theta}\left(\kvect\right)\cdot\Big(\sum_{i\in\Jset}\lambda_{i}\cdot 1_{\left\{ k_{i}=1\right\} }
	+\sum_{i\in\Jset}\supplierrate\cdot p_{i}(\kvect)\cdot 1_{\left\{ k_{i}=0\right\} }\Big)\\
	& =\sum_{i\in\Jset}
	\underbrace{\left(\prod_{\ell=0}^{(\sum_{j\in\Jset}k_{j})-1}\frac{1}{J-\ell}\right)
		\cdot \left(\prod_{j\in \Jset}\left(\frac{1}{\lambda_j}\right)^{k_{j}}\right)
		\cdot \left(\frac{1}{\nu}\right)^{k_{J+1}}}_{\widetilde{\theta}\left(\kvect\right)}\\
	&\cdot \underbrace{\frac{1}{J-\sum_{j\in \Jset}k_{j}}}_{=p_i(\kvect)}
	\cdot \frac{1}{\lambda_{i}}\cdot \nu
	\cdot\lambda_{i}\cdot1_{\left\{ k_{i}=0\right\} }\\
	& \phantomeq+\sum_{i\in\Jset}
	\underbrace{\left(\prod_{\ell=0}^{(\sum_{j\in\Jset}k_{j})-1}\frac{1}{J-\ell}\right)
		\cdot \left(\prod_{j\in \Jset}\left(\frac{1}{\lambda_j}\right)^{k_{j}}\right)
		\cdot \left(\frac{1}{\nu}\right)^{k_{J+1}}}_{\widetilde{\theta}\left(\kvect\right)}\\
	&\cdot \underbrace{\left(J+1-\sum_{j\in\Jset}k_{j}\right)}_{=\frac{1}{p_{i}(\kvect-\evect_{i}+\evect_{J+1})}} 
	\cdot \lambda_i \cdot \frac{1}{\nu}
	\cdot\supplierrate\cdot p_{i}(\kvect-\evect_{i}+\evect_{J+1})\cdot1_{\left\{ k_{i}=1\right\} }
	\end{align*}
	The right hand-side of the last equation is 
	\[
	\sum_{i\in\Jset}\widetilde{\theta}\left(\kvect\right)\supplierrate\cdot p_{i}(\kvect)\cdot 1_{\left\{ k_{i}=0\right\}} 
	+
	\sum_{i\in\Jset}\widetilde{\theta}\left(\kvect\right) \lambda_{i}\cdot 1_{\left\{ k_{i}=1\right\}}.
	\]
\end{proof}

\begin{rem}\label{rem:symmetry-property}
	We make a distinction between homogeneous and heterogeneous locations.
	
	We mean by \textit{homogeneous} locations that the inventories have
	identical base stock levels $b_{1}=b_{2}=\cdots=b_{J}$ and identical
	arrival rates $\lambda_{1}=\lambda_{2}=\cdots=\lambda_{J}>0$. Service
	rates $\mu_{j}(\cdot)>0$, $j\in\Jset$, obey no such restrictions. 
	
	We mean by \textit{heterogeneous} locations that there may be different
	arrival rates $\lambda_{j}>0$, $j\in\Jset$ (and any service rate
	$\mu_{j}(\cdot)>0$, $j\in\Jset$) and for the base stock levels hold
	$b_{1}\geq b_{2}\geq\cdots\geq b_{J}$.
	
	As a consequence of the preceding \prettyref{prop:BS-STATE-homo-hetero-theta}
	the following \textit{symmetry property} for homogeneous locations
	with base stock levels $b_{1}=b_{2}=\cdots=b_{J}=1$ is valid.
	
	For all permutations $\sigma$ of $\left\{ 1,\ldots,J\right\} $ holds
	\[
	\widetilde{\theta}(\overbrace{k_{1},k_{2},\ldots,k_{J}}^{\substack{\text{inventories}\\
			\text{at locations}
		}
	},\overbrace{k_{J+1}}^{\text{supplier}})=\widetilde{\theta}(\overbrace{k_{\sigma(1)},k_{\sigma(2)},\ldots,k_{\sigma(J)}}^{\substack{\text{inventories}\\
			\text{at locations}
		}
	},\overbrace{k_{\sigma J+1}}^{\text{supplier}}).
	\]
	For $b_{1}=b_{2}=\cdots=b_{J}>1$ the global balance equations \prettyref{eq:BSN-STAR2-theta-equation}
	reveal directly that this symmetry property holds in this case as
	well. 
\end{rem}

%-----------------------------------------------------------
\subsubsection{Special case: two locations and $b_{j}>1$, $j\in\Jset$\label{app:BSN-STAR2-iterative-algorithm}\label{app:BSN-STAR2-non-identical}}

In this section, we assume that there are two heterogeneous locations
with base stock levels \textbf{$b_{1}\geq b_{2}$}, where\textbf{
}$b_{1}>1$ and $b_{2}\geq1$ and arrival rates $\lambda_{1},\lambda_{2}>0$.
The state transition diagram for such a system is presented in \prettyref{fig:BSN-STAR2-non-identical-state-transition-diagram}.

\begin{figure}[h]
	\centering{}\includegraphics[width=1\textwidth]{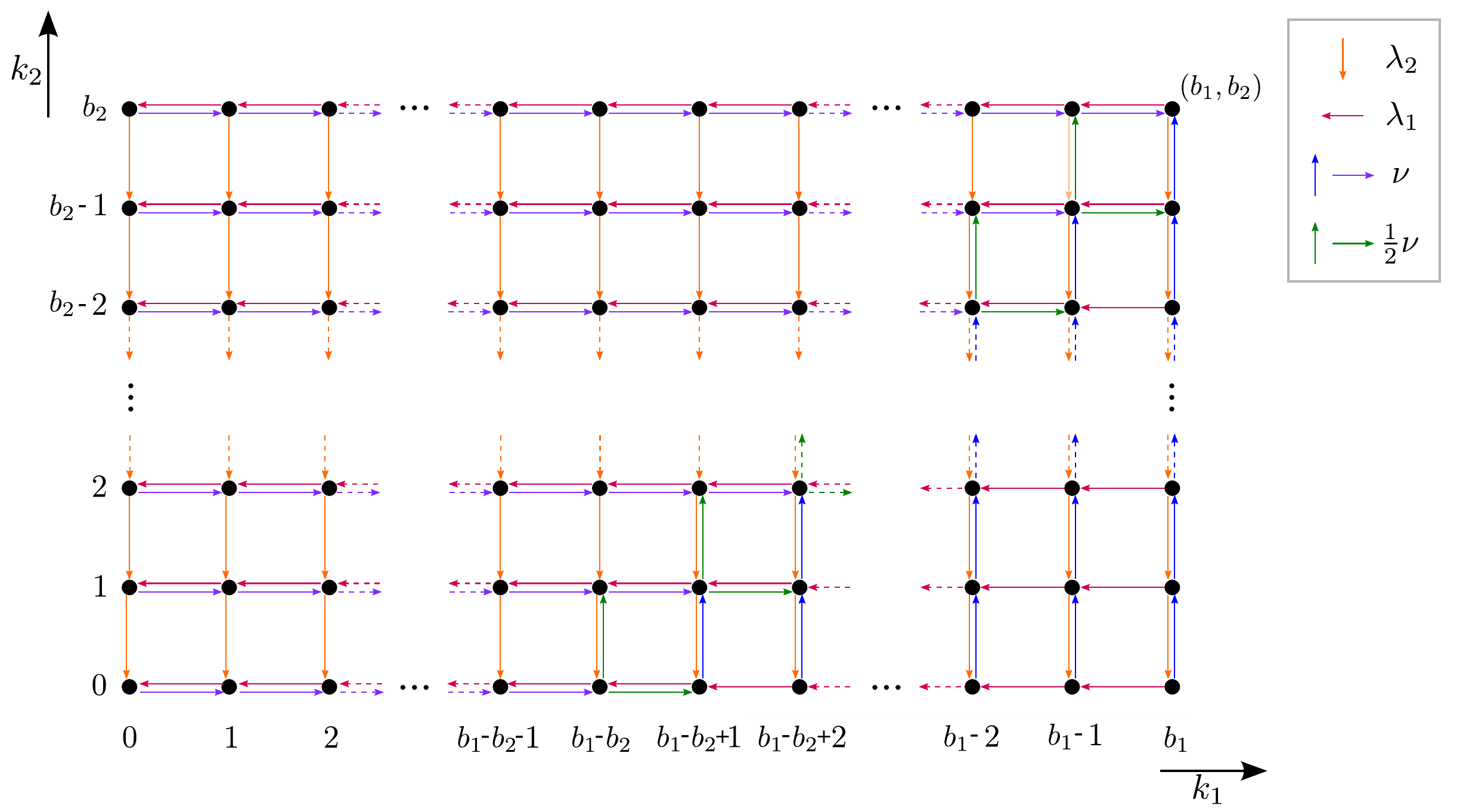}
	\caption{\label{fig:BSN-STAR2-non-identical-state-transition-diagram}
		State transition diagram of a system with two heterogeneous locations}
\end{figure}

To obtain the exact or approximate solutions of queueing models various
analytical, numerical and simulation techniques are available (e.g.~the
power iteration method, generating function approach, product form
solution and recursive solution technique) (cf.~\cite[Section III, pp. 44ff.]{chow;kohler:77}).
The recursive solution technique was first suggested by Herzog, Woo
and Chandy~\cite{herzog;woo;chandy:75}. They demonstrate an efficient
solution for single queueing models with other than exponential arrival
or service time distributions. The recursive technique uses the fact
that the steady state probabilities of the system can sometimes be
expressed in terms of other steady state probabilities. This leads
to a reduction of the number of unknowns in the global balance equations.

Chow and Kohler
present in~\cite{chow;kohler:79} a generalization of the recursive
solution technique. They apply the method to non-homogeneous ($=$
heterogeneous) two-processor systems with special properties in~\cite[Section IV, pp. 358f.]{chow;kohler:79}
and present a sample system using the algorithm in~\cite[Appendix, pp. 360f.]{chow;kohler:79}.
One special property of the load balancing policy that enables to
use their recursive solution technique for the two-processor heterogeneous
systems is that the policy line ($=$ continuous chain of arrival
transitions starting at state $(0,0)$ given that no departure occurs)
partitions the states of the state transition rate diagram into two
regions. Our load balancing policy is slightly different from that
of Chow and Kohler~\cite[Section IV, pp.\ 358f.]{chow;kohler:79}
because of $\frac{1}{2}\nu$ (since if both inventories have the
same difference between the on-hand inventory and the capacity of
the inventory, it enters either with equal probability) as can be
seen in the state transition diagram in \prettyref{fig:BSN-STAR2-non-identical-state-transition-diagram}.
They mentioned that the generalization of recursive solution ``technique
for three or more processors does not appear to be straightforward''~\cite[p. 359]{chow;kohler:79}.

For a system with two heterogeneous locations and \textbf{$b_{1}\geq b_{2}$}
and\textbf{ }$b_{1}>1$, $b_{2}>1$ the recursive method to obtain $\widetilde{\theta}(\kvect)$,
$\kvect\in K$, from the global balance equations $\widetilde{\theta}\cdot\mathbf{Q}_{red}=\mathbf{0}$
(cf. equation \prettyref{eq:BSN-STAR2-theta-equation} on page \pageref{eq:BSN-STAR2-theta-equation}) is described by the following algorithm.
$\kappa$\index{$\kappa$@$\kappa$, variable}
is a variable which represents a temporarily unknown probability and
GBE is used to denote a global balance equation.
We will henceforth use an abbreviated notation because $k_{J+1}=\sum_{j=1}^{J}(b_{j}-k_{j})$
and the base stock levels $b_{j}$, $j\in\Jset$, are fixed parameters:
\[
\widetilde{\theta}\Big(\overbrace{\ \,k_{1},\ \ k_{2}}^{\substack{\text{inventories}\\
		\text{at locations}
	}
}\Big):=\widetilde{\theta}\Big(\overbrace{\ \,k_{1},\ \ k_{2},}^{\substack{\text{inventories}\\
		\text{at locations}
	}
}\overbrace{k_{3}}^{\text{supplier}}\Big)
\]
and hence,
\begin{align*}
p_{i}(\kvect-\evect_{i})  :=p_{i}(\kvect-\evect_{i}+\evect_{J+1}),\qquad
p_{i}(\kvect+\evect_{i})  :=p_{i}(\kvect+\evect_{i}-\evect_{J+1}),\ i\in\Jset.
\end{align*}

A few steps of the  algorithm are visualised in the state transition diagram  in
\prettyref{fig:BS-state-alg-hetero}.
A detailed explanation of the algorithm can be found in \cite[Appendix C.1, pp.~275--311]{otten:17}.

\medskip

\noindent\rule[0.5ex]{1\columnwidth}{1pt}\\
%\begin{small}
\textbf{ALGORITHM }(\textbf{$b_{1}\geq b_{2}$} with\textbf{ }$b_{1}>1$,
$b_{2}>1$ and\textbf{ $\lambda_{1},\lambda_{2}>0$})
\begin{enumerate}
	\item [$\blacktriangleright$]Set $\widetilde{\theta}(b_{1},0)=1$
	\item [$\blacktriangleright$]For $k_{2}=b_{2},\ldots,1$
	
	\begin{enumerate}
		\item [\textbf{(1)}]if $k_{2}=b_{2}$,
		
		\begin{enumerate}
			\item [\textbf{(a)}]set $\widetilde{\theta}(0,b_{2})=\kappa$
			\item [\textbf{(b)}]for $\ell=0,\ldots,b_{2}-2$\\
			use the GBE of state $(b_{1},{\ell})$ \\
			to find an expression for $\widetilde{\theta}(b_{1},{\ell+1})$
			independent of $\kappa$
			\item [\textbf{(c)}]for $k_{1}=0,\ldots,b_{1}-2$\\
			use the GBE of state $({k_{1}},b_{2})$\\
			to find an expression for $\widetilde{\theta}({k_{1}+1},b_{2})$
			as a function of $\kappa$
			\item [\textbf{(d)}]use the GBE of state $(b_{1},{b_{2}})$\\
			to find an expression for $\widetilde{\theta}(b_{1},{b_{2}})$
			as a function of $\kappa$
			\item [\textbf{(e)}]use the GBE of state $(b_{1}-1,{b_{2}})$\\
			to find an expression for $\widetilde{\theta}(b_{1}-1,{b_{2}}-1)$
			as a function of $\kappa$
			\item [\textbf{(f)}]use the GBE of state $(b_{1},b_{2}-1)$ to solve for $\kappa$
			\item [\textbf{(g)}]substitute the value of $\kappa$ into the equations in
			the above steps \textbf{(1)(a) }\\and\textbf{ (1)(c)}-\textbf{(e)}\\
			$ $
		\end{enumerate}
		
		\item [\textbf{ (2)}]if $b_{2}>k_{2}\geq2$ (i.e.\ $b_{1},b_{2}\geq3$),
		
		\begin{enumerate}
			\item [\textbf{ (a)}]set $\widetilde{\theta}(0,k_{2})=\kappa$
			\item [\textbf{ (b)}]for $k_{1}=0,\ldots,b_{1}-(b_{2}-k_{2})-1$
			
			\begin{enumerate}
				\item [\textbf{(i)}]if $k_{1}<b_{1}-(b_{2}-k_{2})-1$,\\
				use the GBE of state $({k_{1}},k_{2})$ \\
				to find an expression for $\widetilde{\theta}({k_{1}+1},k_{2})$
				as a function of $\kappa$
				\item [\textbf{(ii)}]if $k_{1}=b_{1}-(b_{2}-k_{2})-1$,\\
				use the GBE of state $(k_{1},{k_{2}})$ \\
				to find an expression for $\widetilde{\theta}(k_{1},{k_{2}-1})$
				as a function of $\kappa$
			\end{enumerate}
			\item [\textbf{ (c)}]for $\ell=k_{2},\ldots,1$\\
			use the GBE of state $(b_{1}-(b_{2}-k_{2}),{\ell})$ \\
			to find an expression for $\widetilde{\theta}(b_{1}-(b_{2}-k_{2}),{\ell-1})$
			as a function of $\kappa$
			\item [\textbf{ (d)}]use the GBE of state $(b_{1}-(b_{2}-k_{2}),0)$ to solve
			for $\kappa$
			\item [\textbf{ (e)}]substitute the value of $\kappa$ into the equations in
			the above steps \textbf{(2)(a)}-\textbf{(c)}\\
			$ $
		\end{enumerate}
		\item [\textbf{ (3)}]if $b_{2}>k_{2}=1$
		
		\begin{enumerate}
			\item [\textbf{ (a)}]set $\widetilde{\theta}(0,1)=\kappa$
			\item [\textbf{ (b)}]for $k_{1}=0,\ldots,b_{1}-b_{2}+1$
			
			\begin{enumerate}
				\item [\textbf{ (i)}]if $k_{1}<b_{1}-b_{2}$ (if $b_{1}-b_{2}>0$),\\
				use the GBE of state $({k_{1}},1)$ \\
				to find an expression for $\widetilde{\theta}({k_{1}+1},1)$
				as a function of $\kappa$
				\item [\textbf{ (ii)}]if $k_{1}\in\{b_{1}-b_{2},b_{1}-b_{2}+1\}$,\\
				use the GBE of state $(k_{1},{1})$ \\
				to find an expression for $\widetilde{\theta}(k_{1},{0})$
				as a function of $\kappa$
			\end{enumerate}
			\item [\textbf{ (c)}]for $k_{1}=b_{1}-b_{2},\ldots,1$ (if $b_{1}-b_{2}>0$),\\
			use the GBE of state $({k_{1}},0)$\\
			to find an expression for $\widetilde{\theta}({k_{1}-1},0)$
			\item [\textbf{ (d)}]use the GBE of state $(b_{1}-(b_{2}-1),0)$ to solve for
			$\kappa$
			\item [\textbf{ (e)}]substitute the value of $\kappa$ into the equations in
			the above steps\textbf{ (3)(a)}-\textbf{(c)}
		\end{enumerate}
	\end{enumerate}
	\item [$\blacktriangleright$]Normalise all $\widetilde{\theta}(k_{1},k_{2})$
	by setting 
	\[
	\widetilde{\theta}(k_{1},k_{2})\leftarrow\frac{\widetilde{\theta}(k_{1},k_{2})}{\sum_{k_{1}=0}^{b_{1}}\sum_{k_{2}=0}^{b_{2}}\widetilde{\theta}(k_{1},k_{2})}
	\]
\end{enumerate}%\end{small}
\rule[0.5ex]{1\columnwidth}{1pt}

\begin{figure}[H]
	\centering{}\subfigure[First loop with $k_2=b_2$]{\resizebox*{12cm}{!}{
			\includegraphics[width=0.6\textwidth]{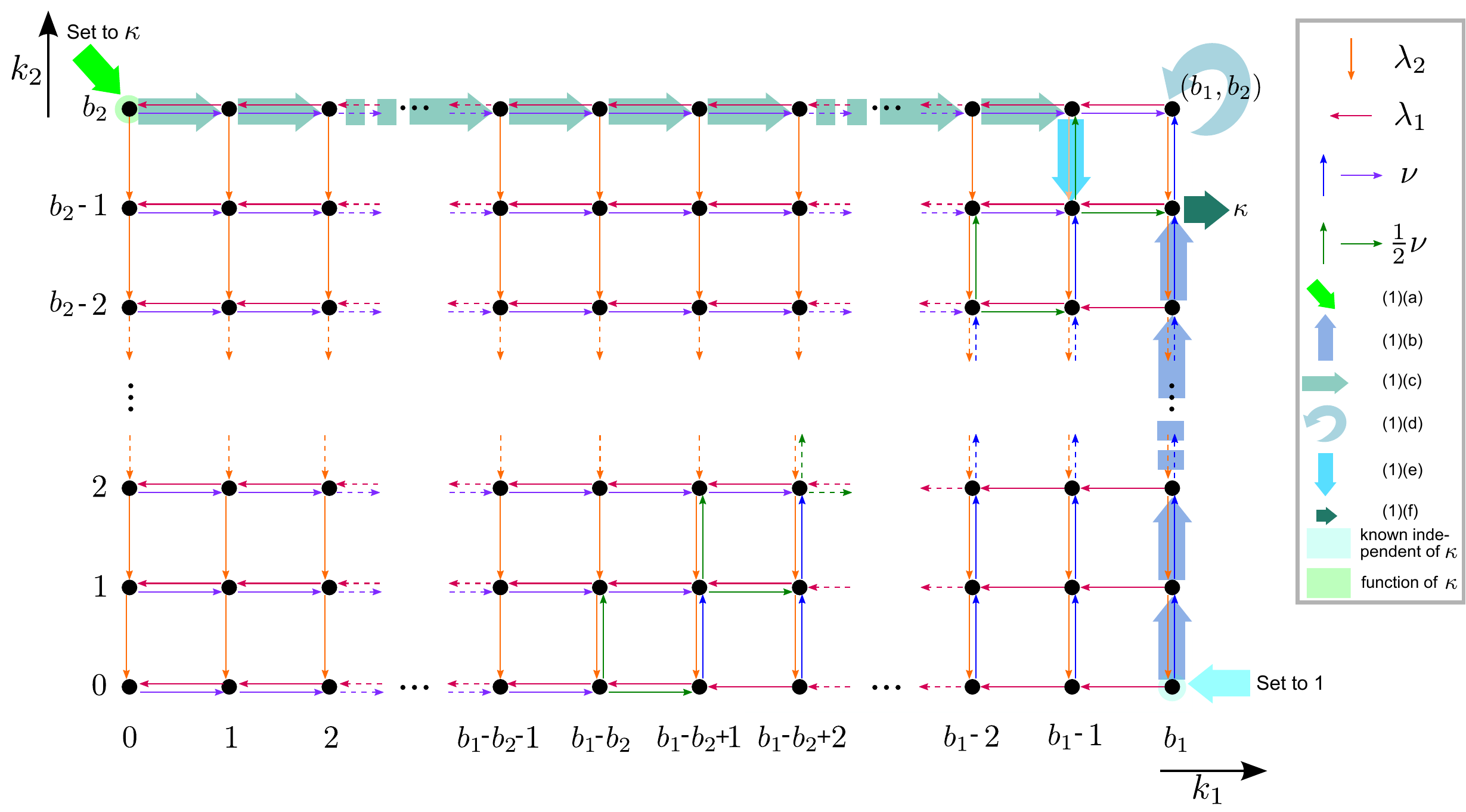}
		}\hspace{6pt} }\\
	\subfigure[Second loop with $k_2=b_2-1$]{\resizebox*{12cm}{!}{ \includegraphics[width=0.6\textwidth]{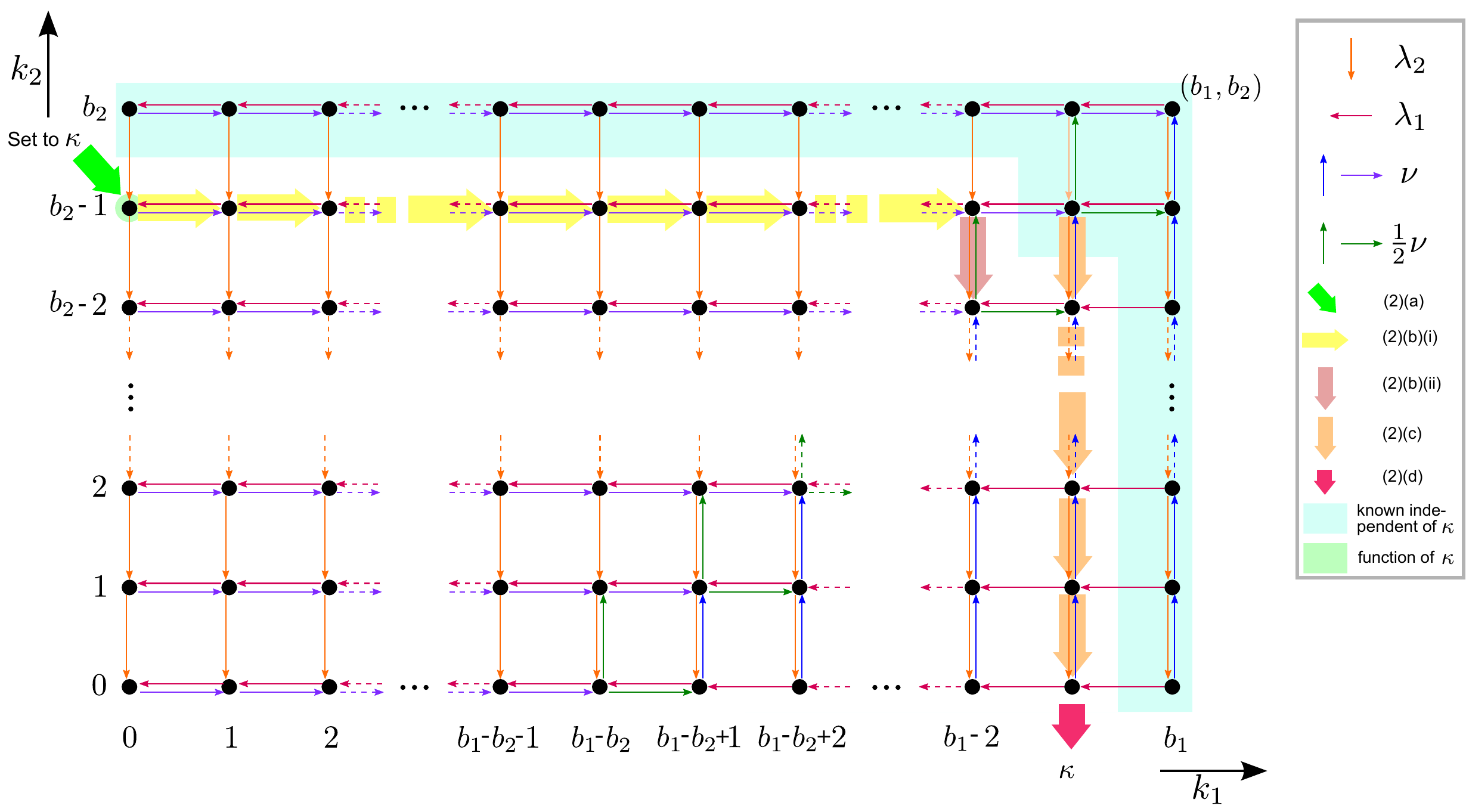}
		}\hspace{6pt} }\\
	\subfigure[Last loop with $k_2=1$]{\resizebox*{12cm}{!}{ \includegraphics[width=0.6\textwidth]{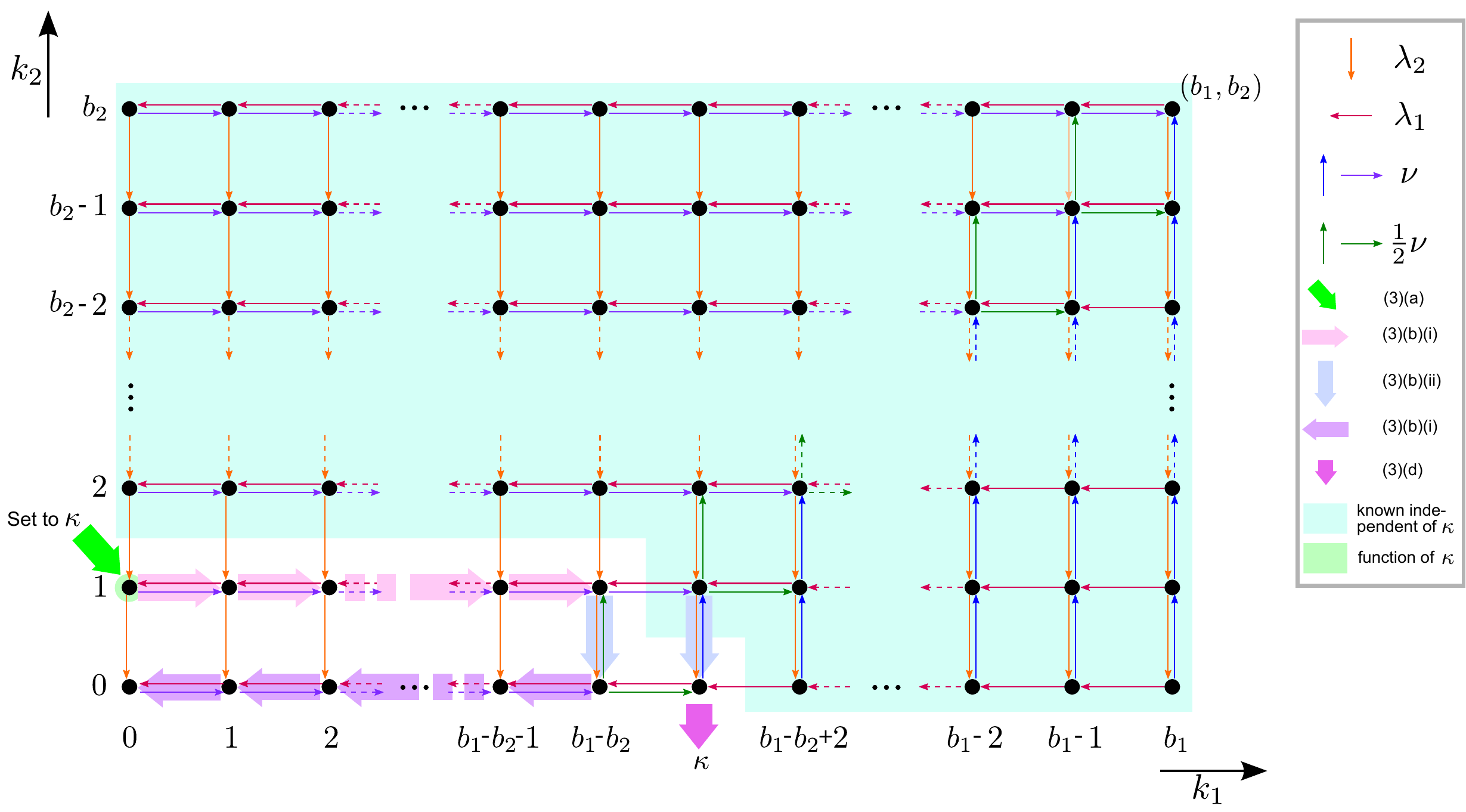}
		}\hspace{6pt} } \caption{\label{fig:BS-state-alg-hetero}Visualisation of a few steps of the
		algorithm}
\end{figure}

\newpage
%-------------------homogeneous---------------------------------------------------
\noindent\textbf{Special case: Two homogeneous locations\label{app:BSN-STAR2-identical}}\\
Our algorithm for two locations can also be applied to obtain $\widetilde{\theta}$
for a system with two homogeneous locations, i.e.~where the inventories
have identical base stock levels $b_{1}=b_{2}>1$ and identical arrival
rates $\lambda_{1}=\lambda_{2}>0$ (and any service rates $\mu_{1},\mu_{2}>0$).
The state transition diagram for such a system is presented in \prettyref{fig:BSN-STAR2-identical-state-transition-diagram-1}. 
However, the algorithm can be simplified for the case of two homogeneous
locations, since in the state transition diagram of a system with
two homogeneous locations is symmetric about the diagonal elements
$(k_{1},k_{2})$. Hence, the state transition diagram can be folded
to obtain a triangle as shown in \prettyref{fig:BSN-STAR2-identical-state-transition-diagram-1}.
Consequently, it holds $\widetilde{\theta}(k_{1},k_{2})=\widetilde{\theta}(k_{2},k_{1})$.

\begin{figure}[H]
	\centering{}\includegraphics[width=1\textwidth]{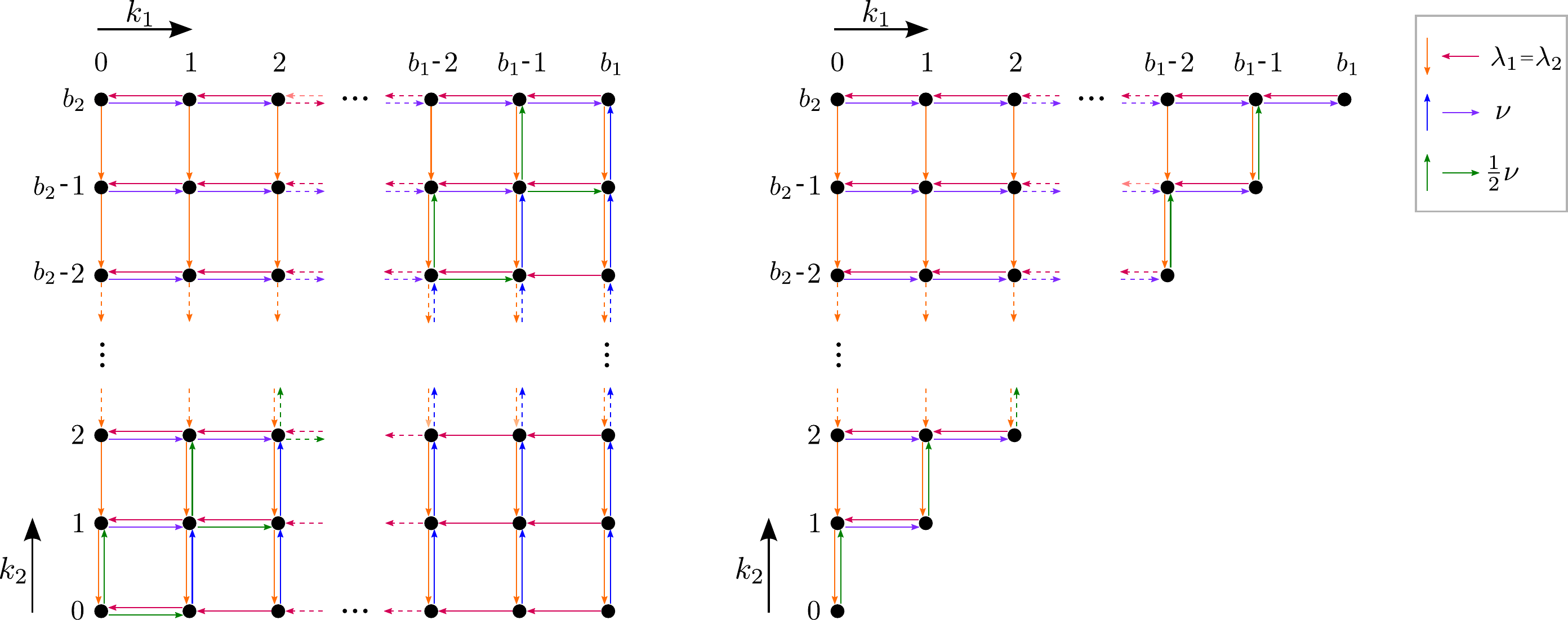}\caption{\label{fig:BSN-STAR2-identical-state-transition-diagram-1}State transition
		diagram (left) and folded state transistion diagram (right) of a system with two homogeneous locations}
\end{figure}

Chow and Kohler analyse in \cite{chow;kohler:77} the performance
of homogeneous two-processor distributed computer systems under several
dynamic load balancing policies. Their analy-\linebreak{}
sis is based on the recursive solution technique and they illustrated
the algorithm for a sample system~\cite[Appendix, pp. 51f.]{chow;kohler:77}.
Their strategy ``join the shorter queue without channel transfer''
in Chow and Kohler's Model B~\cite[pp.\ 42f.]{chow;kohler:77} is
equivalent to our strict load balancing policy. 
Additionally, they allow a channel transfer in Model C~\cite[pp. 42f.]{chow;kohler:77}
and show that under heavy load this strategy can improve the performance
(turnaround time) of the two-processor system (cf.~\cite[pp. 47f.]{chow;kohler:77}).

We can also extend our model with two locations by a channel transfer
and obtain a stationary distribution of product form. Then, the inventories
at the locations are connected through a transfer channel. The transfer
time of the channel is exponentially distributed with rate $\beta>0$.
If the difference between the number of items in location $i$ and
the number of items at location $j$ is equal to or greater than two
($=$ disbalance condition), then the transfer channel initiates
a transfer from items from location $i$ to location $j$. There can
be only one transfer at a time and the transfer of an item is discontinued
if the disbalance condition changes before the transfer is completed.
The state transition diagram can still be folded into a triangle,
so that the symmetry property holds. Second, because the channel transfer
leads to a further summand on the left side of the GBE (flow into
the state). The state transition diagram as well as the folded state
transition diagram for such a system with two locations are presented
in \prettyref{fig:BSN-STAR2-identical-state-transition-diagram-2}.

\begin{figure}[H]
	\centering{}\includegraphics[width=1\textwidth]{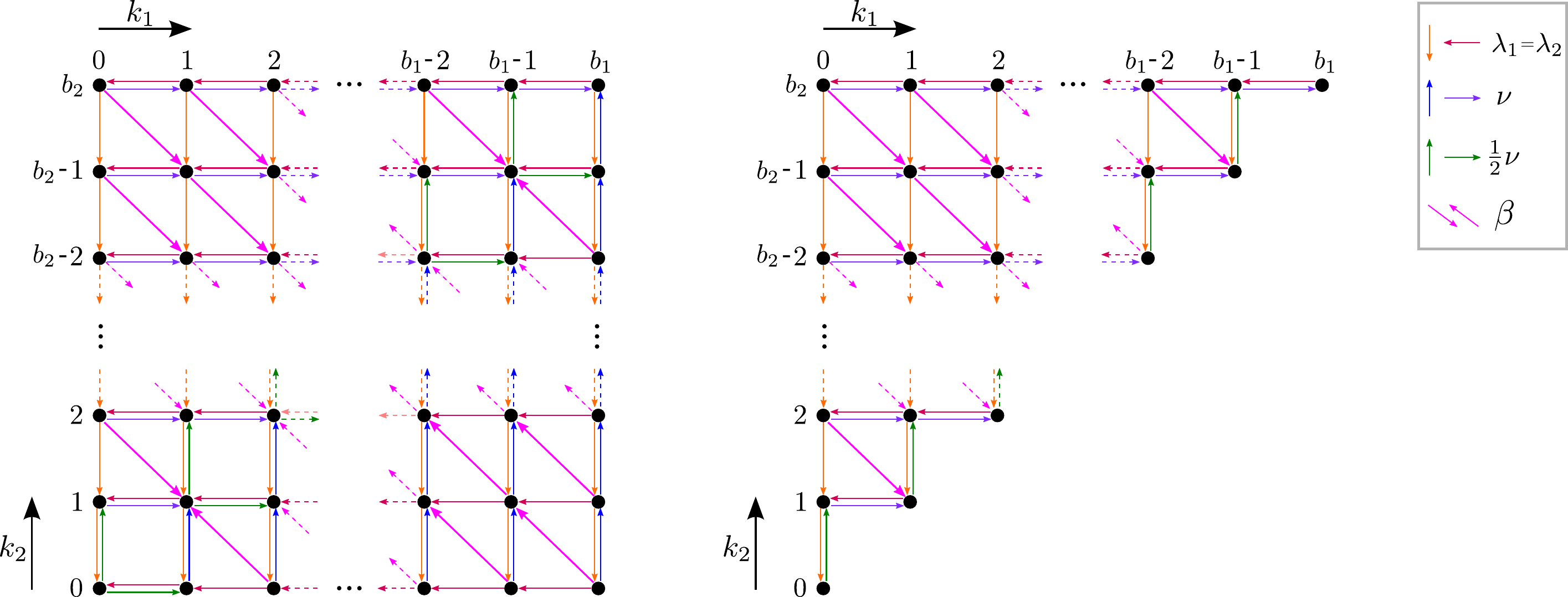}\caption{\label{fig:BSN-STAR2-identical-state-transition-diagram-2}State transition
		diagram (left) and folded state transition diagram (right) of a system with two homogeneous locations and transfer channel}
\end{figure}

%---------------------------------------------------------------------------------
\section{Structural properties of the stationary inventory-replenishment subsystem \label{sec:strucural properties}}

In this section, we assume that the queueing-inventory process $Z$
is ergodic. We make again a distinction between homogeneous and heterogeneous
locations.
In this section, we will use an abbreviated notation because $k_{J+1}=\sum_{j=1}^{J}(b_{j}-k_{j})$
and the base stock levels $b_{j}$, $j\in\Jset$, are fixed parameters:
\[
\theta(\overbrace{k_{1},k_{2},\ldots,k_{J}}^{\substack{\text{inventories}\\
		\text{at locations}
	}
}):=\theta(\overbrace{k_{1},k_{2},\ldots,k_{J}}^{\substack{\text{inventories}\\
		\text{at locations}
	}
},\overbrace{k_{J+1}}^{\text{supplier}}).
\]
$ $

\paragraph{Homogeneous locations\label{par:BSN-STAR2-homo-structure}}

Let $(Y_{1},Y_{2},\ldots Y_{J},W_{J+1})$ be a random variable which
is distributed according to the marginal steady state probability
for the inventory-replenishment subsystem.\footnote{It should be noted that $\theta(k_{1},k_{2},\ldots,k_{J},k_{J+1})=P(Y_{1}=k_{1},Y_{2}=k_{2},\ldots,Y_{J}=k_{J},W_{J+1}=k_{J+1})=P(Y_{1}=k_{1},Y_{2}=k_{2},\ldots,Y_{J}=k_{J})$
	because the base stock levels $b_{j}$, $j\in\Jset$, are fixed parameters
	and $k_{J+1}=\sum_{j=1}^{J}(b_{j}-k_{j})$.}
\begin{prop}
	For the inventory process holds
	\[
	P\left(Y_{1}=\ell\right)=P\left(Y_{j}=\ell\right),\quad j\in\Jset\setminus\left\{ 1\right\} ,\ \ell\in\{0,\ldots,b_{1}\},
	\]
	\begin{align}
	& \phantomeq P\left(Y_{1}=\ell\right)\cdot\lambda_{1}\label{eq:homo-inventory-process}\\
	& =P\left(Y_{j}=\ell-1\ \text{for }j\in\Jset\right)\cdot\frac{\nu}{J}\nonumber \\
	& \phantomeq+\sum_{i=1}^{J-1}P\left(Y_{j}=\ell-1\ \text{for }j=1,\ldots,i\ \text{and }\ell-1<Y_{k}\leq b_{k}\ \text{for }k=i+1,\ldots,J\right)\nonumber \\
	& \phantomeq\hphantom{+\sum_{i=1}^{J-1}}\cdot\binom{J-1}{i-1}\cdot\frac{\nu}{i},\qquad\qquad\qquad\ell\in\{1,\ldots,b_{1}\}.\nonumber 
	\end{align}
\end{prop}
\begin{proof}
	It holds
	\[
	P\left(Y_{1}=\ell\right)=P\left(Y_{j}=\ell\right),\quad j\in\Jset\setminus\left\{ 1\right\} ,\ \ell\in\{0,\ldots,b_{1}\}
	\]
	because of the symmetry property for homogeneous locations (see~\prettyref{rem:symmetry-property})
	%\[
	%\theta(k_{1},k_{2},\ldots,k_{J})=\theta(k_{2},k_{1},\ldots,k_{J})=\cdots=\theta(k_{J},k_{J-1},\ldots,k_{1})
	%\]
	and because
	\[
	P\left(Y_{1}=\ell\right)=\sum_{k_{2}=0}^{b_{2}}\cdots\sum_{k_{J}=0}^{b_{J}}\theta(\ell,k_{2},\ldots,k_{J}).
	\]
	
	The equation \prettyref{eq:homo-inventory-process} can be proven
	by the cut-criterion for positive recurrent processes (see \cite[Lemma
	1.4, p. 8]{kelly:79}).
	
	For $\ell\in\left\{ 1,\ldots,b_{1}\right\} $, it can be proven by
	a cut, which divides $E$ into complementary sets according to the
	size of the inventory at location $1$ that is less than or equal
	to $\ell-1$ or greater than $\ell-1$, i.e.\ into the sets
	\[
	\Big\{(k_{1},k_{2},\ldots,k_{J}):k_{1}\in\{0,1,\ldots,\ell-1\},\:k_{j}\in\{0,\ldots,b_{j}\},\ j\in\left\{ 2,\ldots,J\right\} \Big\},\hphantom{\quad\ell_{1}\in\left\{ 1,\ldots,b_{1}\right\} .}
	\]
	\[
	\left\{ (\widetilde{k}_{1},\widetilde{k}_{2},\ldots,\widetilde{k}_{J}):\widetilde{k}_{1}\in\{\ell,\ldots,b_{1}\},\:\widetilde{k}_{j}\in\{0,\ldots,b_{j}\},\ j\in\left\{ 2,\ldots,J\right\} \right\} ,\ \ell\in\left\{ 1,\ldots,b_{1}\right\} .
	\]
	Then, it follows for $\ell\in\left\{ 1,\ldots,b_{1}\right\} $
	\begin{align}
	& \phantomeq\sum_{k_{1}=0}^{\ell-1}\sum_{k_{2}=0}^{b_{2}}\cdots\sum_{k_{J}=0}^{b_{J}}\ \sum_{\widetilde{k}_{1}=\ell}^{b_{1}}\sum_{\widetilde{k}_{2}=0}^{b_{2}}\cdots\sum_{\widetilde{k}_{J}=0}^{b_{J}}\theta(k_{1},k_{2},\ldots,k_{J})\nonumber \\
	& \hphantom{\phantomeq\sum_{k_{1}=0}^{\ell-1}\sum_{k_{2}=0}^{b_{2}}\cdots\sum_{k_{J}=0}^{b_{J}}\ \sum_{\widetilde{k}_{1}=\ell}^{b_{1}}\sum_{\widetilde{k}_{2}=0}^{b_{2}}\cdots\sum_{\widetilde{k}_{J}=0}^{b_{J}}}\cdot q_{red}\left((k_{1},k_{2},\ldots,k_{J});(\widetilde{k}_{1},\widetilde{k}_{2},\ldots,\widetilde{k}_{J})\right)\nonumber \\
	& =\sum_{\widetilde{k}_{1}=\ell}^{b_{1}}\sum_{\widetilde{k}_{2}=0}^{b_{2}}\cdots\sum_{\widetilde{k}_{J}=0}^{b_{J}}\ \sum_{k_{1}=0}^{\ell-1}\sum_{k_{2}=0}^{b_{2}}\cdots\sum_{k_{J}=0}^{b_{J}}\theta(\widetilde{k}_{1},\widetilde{k}_{2},\ldots,\widetilde{k}_{J})\nonumber \\
	& \hphantom{\phantomeq\sum_{k_{1}=0}^{\ell-1}\sum_{k_{2}=0}^{b_{2}}\cdots\sum_{k_{J}=0}^{b_{J}}\ \sum_{\widetilde{k}_{1}=\ell}^{b_{1}}\sum_{\widetilde{k}_{2}=0}^{b_{2}}\cdots\sum_{\widetilde{k}_{J}=0}^{b_{J}}}\cdot q_{red}\left((\widetilde{k}_{1},\widetilde{k}_{2},\ldots,\widetilde{k}_{J});(k_{1},k_{2},\ldots,k_{J})\right)\nonumber \\
	\nonumber \\
	\Leftrightarrow & \phantomeq\theta(\ell-1,\ell-1,\ldots,\ell-1)\cdot\frac{\nu}{J}\label{eq:load-formel-1}\\
	& \phantomeq+\sum_{i=1}^{J-1}\sum_{k_{i+1}=\ell}^{b_{i+1}}\cdots\sum_{k_{J}=\ell}^{b_{J}}\theta(\ell-1,\ell-1,\ldots,\ell-1,k_{i+1},\ldots,k_{J})\cdot\binom{J-1}{i-1}\cdot\frac{1}{i}\cdot\nu\label{eq:load-formel-2}\\
	& =\sum_{\widetilde{k}_{2}=0}^{b_{2}}\cdots\sum_{\widetilde{k}_{J}=0}^{b_{J}}\theta(\ell,\widetilde{k}_{2},\ldots,\widetilde{k}_{J})\cdot\lambda_{1}.\nonumber 
	\end{align}
	The only possible transitions from the set, where the size of the
	inventory at location $1$ is less than or equal to $\ell-1$, to
	the set, where the inventory at location $1$ is greater than $\ell-1$,
	are transitions according to a replenishment. In particular, transitions
	from\\
	$
	\Big\{(\ell-1,k_{2},\ldots,k_{J}):k_{j}\in\{\ell-1,\ldots,b_{j}\},\ j\in\left\{ 2,\ldots,J\right\} \Big\},\hphantom{\quad\ell_{1}\in\left\{ 1,\ldots,b_{1}\right\} .}
	$
	\\
	to
	$
	\left\{ (\ell,\widetilde{k}_{2},\ldots,\widetilde{k}_{J}):\widetilde{k}_{j}\in\{\ell-1,\ldots,b_{j}\},\ j\in\left\{ 2,\ldots,J\right\} \right\} ,\ \ell\in\left\{ 1,\ldots,b_{1}\right\} .
	$
	
	A replenishment at location $1$ is only possible if $\{1\}\subseteq\underset{j\in\Jset}{\arg\max}(b_{j}-k_{j})$.
	This means that there is no other location with higher difference
	between the on-hand inventory and the capacity of the inventory ($=$
	base stock level). Consequently, all possible states where the other
	locations have $\ell-1$ items or more items in the inventory have
	to be considered.\\
	In \prettyref{eq:load-formel-1} all locations have exactly $\ell-1$
	items in the inventory and in \prettyref{eq:load-formel-2} $i$ states
	how many locations have exactly $\ell-1$ items in the inventory.
	This results in the factor $\frac{1}{i}$, which is the probability
	that the finished item is sent to location $i$. The symmetry property
	leads to the factor $\binom{J-1}{i-1}$.
	Hence, we have shown for $\ell\in\left\{ 1,\ldots,b_{1}\right\} $
	\begin{align*}
	& \phantomeq P\left(Y_{1}=\ell\right)\cdot\lambda_{1}\\
	& =P\left(Y_{j}=\ell-1\ \text{for }j\in\Jset\right)\cdot\frac{\nu}{J}\\
	& \phantomeq+\sum_{i=1}^{J-1}P\left(Y_{j}=\ell-1\ \text{for }j=1,\ldots,i\ \text{and }\ell-1<Y_{k}\leq b_{k}\ \text{for }k=i+1,\ldots,J\right)\\
	& \phantomeq\hphantom{+\sum_{i=1}^{J-1}}\cdot\binom{J-1}{i-1}\cdot\frac{\nu}{i}.
	\end{align*}
\end{proof}
\paragraph{Heterogeneous locations\label{par:BSN-STAR2-homo-structure-1}}
Let $(Y_{1},Y_{2},W_{3})$ be a random variable which is distributed
according to the marginal steady state probability for the inventory-replenishment
subsystem.\footnote{It should be noted that $\theta(k_{1},k_{2},k_{3})=P(Y_{1}=k_{1},Y_{2}=k_{2},W_{3}=k_{3})=P(Y_{1}=k_{1},Y_{2}=k_{2})$,
	because the base stock levels $b_{1}$ and $b_{2}$ are fixed parameters
	and $k_{3}=(b_{1}+b_{2})-(k_{1}+k_{2})$.}
\begin{rem}
	From equation \prettyref{eq:BS-state-homo-Yl1} in the following proposition
	follows
	\[
	P(Y_{1}=\ell_{1})=P(Y_{1}=0)\cdot\left(\frac{\nu}{\lambda_{1}}\right)^{\ell_{1}},\qquad\ell_{1}\in\left\{ 1,\ldots,b_{1}-b_{2}\right\} .
	\]
\end{rem}

\begin{prop}
	For the inventory process holds\\
	for $\ell_{1}=1,\ldots,b_{1}-b_{2}$
	\begin{equation}
	P(Y_{1}=\ell_{1})\cdot\lambda_{1}=P(Y_{1}=\ell_{1}-1)\cdot\nu,\qquad\qquad\qquad\qquad\qquad\ \label{eq:BS-state-homo-Yl1}
	\end{equation}
	for $\ell_{1}=b_{1}-b_{2}+1,\ldots,b_{1}-1$ 
	\begin{align}
	P(Y_{1}=\ell_{1})\cdot\lambda_{1} & =P(Y_{1}=\ell_{1}-1,Y_{2}=\ell_{1}-1-b_{1}+b_{2})\cdot\frac{1}{2}\nu\nonumber \\
	& \phantomeq+P(Y_{1}=\ell_{1}-1,Y_{2}>\ell_{1}-1-b_{1}+b_{2})\cdot\nu,\label{eq:BS-state-homo-Yl2}
	\end{align}
	for $\ell_{1}=b_{1}$
	\begin{align}
	P(Y_{1}=b_{1})\cdot\lambda_{1} & =P(Y_{1}=b_{1}-1,Y_{2}=b_{2}-1)\cdot\frac{1}{2}\nu\qquad\qquad\nonumber \\
	& \phantomeq+P(Y_{1}=b_{1}-1,Y_{2}=b_{2})\cdot\nu,\label{eq:BS-state-homo-Yl3}
	\end{align}
	\textup{for $\ell_{2}=1,\ldots,b_{2}$ 
		\begin{align}
		P(Y_{2}=\ell_{2})\cdot\lambda_{2} & =P(Y_{1}=b_{1}-b_{2}+\ell_{2}-1,Y_{2}=\ell_{2}-1)\cdot\frac{1}{2}\nu\nonumber \\
		& \phantomeq+P(Y_{1}>b_{1}-b_{2}+\ell_{2}-1,Y_{2}=\ell_{2}-1)\cdot\nu.\label{eq:BS-state-homo-Yl4}
		\end{align}
	}
\end{prop}

\begin{proof}
	The equations can be proven by the cut-criterion for positive recurrent
	processes (see \cite[Lemma	1.4, p. 8]{kelly:79}). 
	
	For $\ell_{1}\in\left\{ 1,\ldots,b_{1}-b_{2}\right\} $, equation
	\prettyref{eq:BS-state-homo-Yl1} can be proven by a cut, which divides
	$E$ into complementary sets according to the size of the inventory
	at location $1$ that is less than or equal to $\ell_{1}-1$ or greater
	than $\ell_{1}-1$, i.e.\ into the sets
	\[
	\Big\{(k_{1},k_{2}):k_{1}\in\{0,1,\ldots,\ell_{1}-1\},\:k_{2}\in\{0,\ldots,b_{2}\}\Big\},\hphantom{\quad\ell_{1}\in\left\{ 1,\ldots,b_{1}-b_{2}\right\} .}
	\]
	\[
	\left\{ (\widetilde{k}_{1},\widetilde{k}_{2}):\widetilde{k}_{1}\in\{\ell_{1},\ldots,b_{1}\},\:\widetilde{k}_{2}\in\{0,\ldots,b_{2}\}\right\} ,\quad\ell_{1}\in\left\{ 1,\ldots,b_{1}-b_{2}\right\} .
	\]
	Then, it follows for $\ell_{1}\in\left\{ 1,\ldots,b_{1}-b_{2}\right\} $
	\begin{align*}
	& \phantomeq\sum_{k_{1}=0}^{\ell_{1}-1}\sum_{k_{2}=0}^{b_{2}}\ \sum_{\widetilde{k}_{1}=\ell_{1}}^{b_{1}}\sum_{\widetilde{k}_{2}=0}^{b_{2}}\theta(k_{1},k_{2})\cdot q_{red}\left((k_{1},k_{2});(\widetilde{k}_{1},\widetilde{k}_{2})\right)\\
	& =\sum_{\widetilde{k}_{1}=\ell_{1}}^{b_{1}}\sum_{\widetilde{k}_{2}=0}^{b_{2}}\ \sum_{k_{1}=0}^{\ell_{1}-1}\sum_{k_{2}=0}^{b_{2}}\theta(\widetilde{k}_{1},\widetilde{k}_{2})\cdot q_{red}\left((\widetilde{k}_{1},\widetilde{k}_{2});(k_{1},k_{2})\right)\\
	\Leftrightarrow\quad & \phantomeq\underbrace{\sum_{k_{2}=0}^{b_{2}}\theta(\ell_{1}-1,k_{2})}_{=P(Y_{1}=\ell_{1}-1)}\cdot\nu=\underbrace{\sum_{\widetilde{k}_{2}=0}^{b_{2}}\theta(\ell_{1},\widetilde{k}_{2})}_{=P(Y_{1}=\ell_{1})}\cdot\lambda_{1}.
	\end{align*}
	Hence, we have shown for $\ell_{1}\in\left\{ 1,\ldots,b_{1}-b_{2}\right\} $
	\begin{align*}
	P(Y_{1} & =\ell_{1}-1)\cdot\nu=P(Y_{1}=\ell_{1})\cdot\lambda_{1}.
	\end{align*}
	For $\ell_{1}\in\left\{ b_{1}-b_{2}+1,\ldots,b_{1}-1\right\} $, equation
	\prettyref{eq:BS-state-homo-Yl2} can be proven by a cut, which divides
	$E$ into complementary sets according to the size of the inventory
	at location $1$ that is less than or equal to $\ell_{1}-1$ or greater
	than $\ell_{1}-1$, i.e.\ into the sets
	\[
	\left\{ (k_{1},k_{2}):k_{1}\in\{0,1,\ldots,\ell_{1}-1\},\:k_{2}\in\{0,\ldots,b_{2}\}\right\} ,\hphantom{\quad\ell_{1}\in\left\{ b_{1}-b_{2}+1,\ldots,b_{1}-1\right\} .}
	\]
	\[
	\left\{ (\widetilde{k}_{1},\widetilde{k}_{2}):\widetilde{k}_{1}\in\{\ell_{1},\ldots,b_{1}\},\:\widetilde{k}_{2}\in\{0,\ldots,b_{2}\}\right\} ,\quad\ell_{1}\in\left\{ b_{1}-b_{2}+1,\ldots,b_{1}-1\right\} .
	\]
	Then, it follows for $\ell_{1}\in\left\{ b_{1}-b_{2}+1,\ldots,b_{1}-1\right\} $
	\begin{align*}
	& \phantomeq\sum_{k_{1}=0}^{\ell_{1}-1}\sum_{k_{2}=0}^{b_{2}}\ \sum_{\widetilde{k}_{1}=\ell_{1}}^{b_{1}}\sum_{\widetilde{k}_{2}=0}^{b_{2}}\theta(k_{1},k_{2})\cdot q_{red}\left((k_{1},k_{2});(\widetilde{k}_{1},\widetilde{k}_{2})\right)\\
	& =\sum_{\widetilde{k}_{1}=\ell_{1}}^{b_{1}}\sum_{\widetilde{k}_{2}=0}^{b_{2}}\ \sum_{k_{1}=0}^{\ell_{1}-1}\sum_{k_{2}=0}^{b_{2}}\theta(\widetilde{k}_{1},\widetilde{k}_{2})\cdot q_{red}\left((\widetilde{k}_{1},\widetilde{k}_{2});(k_{1},k_{2})\right)\\
	\Leftrightarrow\quad & \underbrace{\sum_{k_{2}=b_{2}-(b_{1}-\ell_{1})}^{b_{2}}\theta(\ell_{1}-1,k_{2})\cdot}_{=P(Y_{1}=\ell_{1}-1,Y_{2}>b_{2}-\left(b_{1}-(\ell_{1}-1)\right)}\nu+\underbrace{\vphantom{\sum_{k_{2}=}^{b_{2}}}\theta(\ell_{1}-1,b_{2}-\left(b_{1}-(\ell_{1}-1)\right))}_{=P(Y_{1}=\ell_{1}-1,Y_{2}=b_{2}-\left(b_{1}-(\ell_{1}-1)\right)}\cdot\frac{1}{2}\nu\\
	&=\underbrace{\sum_{\widetilde{k}_{2}=0}^{b_{2}}\theta(\ell_{1},\widetilde{k}_{2})}_{=P(Y_{1}=\ell_{1})}\cdot\lambda_{1}.
	\end{align*}
	Hence, we have shown for $\ell_{1}\in\left\{ b_{1}-b_{2}+1,\ldots,b_{1}-1\right\} $
	\begin{align*}
	P(Y_{1}=\ell_{1})\cdot\lambda_{1} & =P(Y_{1}=\ell_{1}-1,Y_{2}=b_{2}-\left(b_{1}-(\ell_{1}-1)\right))\cdot\frac{1}{2}\nu\\
	& \phantomeq+P(Y_{1}=\ell_{1}-1,Y_{2}>b_{2}-\left(b_{1}-(\ell_{1}-1)\right))\cdot\nu\\
	& =P(Y_{1}=\ell_{1}-1,Y_{2}=\ell_{1}-1-b_{1}+b_{2})\cdot\frac{1}{2}\nu\\
	& \phantomeq+P(Y_{1}=\ell_{1}-1,Y_{2}>\ell_{1}-1-b_{1}+b_{2})\cdot\nu.\\
	\end{align*}
	For $\ell_{1}=b_{1}$, equation \prettyref{eq:BS-state-homo-Yl3}
	can be proven by a cut, which divides $E$ into complementary sets
	according to the size of the inventory at location $1$ that is less
	than or equal to $b_{1}-1$ or greater than $b_{1}-1$, i.e.\ into
	the sets
	\[
	\Big\{(k_{1},k_{2}):k_{1}\in\{0,1,\ldots,b_{1}-1\},\:k_{2}\in\{0,\ldots,b_{2}\}\Big\},
	\]
	\[
	\left\{ (\widetilde{k}_{1},\widetilde{k}_{2}):\widetilde{k}_{1}=b_{1},\:\widetilde{k}_{2}\in\{0,\ldots,b_{2}\}\right\} .
	\]
	Then, it follows for
	\begin{align*}
	& \phantomeq\sum_{k_{1}=0}^{b_{1}-1}\sum_{k_{2}=0}^{b_{2}}\ \sum_{\widetilde{k}_{1}=b_{1}}^{b_{1}}\sum_{\widetilde{k}_{2}=0}^{b_{2}}\theta(k_{1},k_{2})\cdot q_{red}\left((k_{1},k_{2});(\widetilde{k}_{1},\widetilde{k}_{2})\right)\\
	& =\sum_{\widetilde{k}_{1}=b_{1}}^{b_{1}}\sum_{\widetilde{k}_{2}=0}^{b_{2}}\ \sum_{k_{1}=0}^{b_{1}-1}\sum_{k_{2}=0}^{b_{2}}\theta(\widetilde{k}_{1},\widetilde{k}_{2})\cdot q_{red}\left((\widetilde{k}_{1},\widetilde{k}_{2});(k_{1},k_{2})\right)\\
	\Leftrightarrow\quad & \phantomeq\underbrace{\theta(b_{1}-1,b_{2})}_{=P(Y_{1}=b_{1}-1,Y_{2}=b_{2})}\cdot\nu+\underbrace{\theta(b_{1}-1,b_{2}-1)}_{=P(Y_{1}=b_{1}-1,Y_{2}=b_{2}-1)}\cdot\frac{1}{2}\nu=\underbrace{\sum_{\widetilde{k}_{2}=0}^{b_{2}}\theta(b_{1},\widetilde{k}_{2})}_{=P(Y_{1}=b_{1})}\cdot\lambda_{1}.
	\end{align*}
	Hence, we have shown
	\begin{align*}
	P(Y_{1}=b_{1})\cdot\lambda_{1} & =P(Y_{1}=b_{1}-1,Y_{2}=b_{2}-1)\cdot\frac{1}{2}\nu
	+P(Y_{1}=b_{1}-1,Y_{2}=b_{2})\cdot\nu.
	\end{align*}
	
	For $\ell_{2}\in\left\{ 1,\ldots,b_{2}\right\} $, equation \prettyref{eq:BS-state-homo-Yl4}
	can be proven by a cut, which divides $E$ into complementary sets
	according to the size of the inventory at location $2$ that is less
	than or equal to $\ell_{2}-1$ or greater than $\ell_{2}-1$, i.e.\ into
	the sets
	\[
	\Big\{(k_{1},k_{2}):k_{1}\in\{0,\ldots,b_{1}\},\:k_{2}\in\{0,1,\ldots,\ell_{2}-1\}\Big\},\hphantom{\quad\ell_{1}\in\left\{ 1,\ldots,b_{2}\right\} .}
	\]
	\[
	\left\{ (\widetilde{k}_{1},\widetilde{k}_{2}):\widetilde{k}_{1}\in\{0,\ldots,b_{1}\},\:\widetilde{k}_{2}\in\{\ell_{2},\ldots,b_{2}\}\right\} ,\quad\ell_{2}\in\left\{ 1,\ldots,b_{2}\right\} .
	\]
	Then, it follows for $\ell_{2}\in\left\{ 1,\ldots,b_{2}\right\} $
	\begin{align*}
	& \phantomeq\sum_{k_{1}=0}^{b_{1}}\sum_{k_{2}=0}^{\ell_{2}-1}\ \sum_{\widetilde{k}_{1}=0}^{b_{1}}\sum_{\widetilde{k}_{2}=\ell_{2}}^{b_{2}}\theta(k_{1},k_{2})\cdot q_{red}\left((k_{1},k_{2});(\widetilde{k}_{1},\widetilde{k}_{2})\right)\\
	& =\sum_{\widetilde{k}_{1}=0}^{b_{1}}\sum_{\widetilde{k}_{2}=\ell_{2}}^{b_{2}}\ \sum_{k_{1}=0}^{b_{1}}\sum_{k_{2}=0}^{\ell_{2}-1}\theta(\widetilde{k}_{1},\widetilde{k}_{2})\cdot q_{red}\left((\widetilde{k}_{1},\widetilde{k}_{2});(k_{1},k_{2})\right)\\
	\Leftrightarrow\quad & \underbrace{\vphantom{\sum_{\widetilde{k}_{1}=0}^{b_{1}}}\theta(b_{1}-\left(b_{2}-(\ell_{2}-1)\right),\ell_{2}-1)}_{=P(Y_{1}=b_{1}-\left(b_{2}-(\ell_{2}-1)\right),Y_{2}=\ell_{2}-1)}\cdot\frac{1}{2}\nu\underbrace{\sum_{k_{1}=b_{1}-(b_{2}-\ell_{2})}^{b_{1}}\theta(k_{1},\ell_{2}-1)}_{=P(Y_{1}>b_{1}-\left(b_{2}-(\ell_{2}-1)\right),Y_{2}=\ell_{2}-1)}\cdot\nu\\
	&=\underbrace{\sum_{\widetilde{k}_{1}=0}^{b_{1}}\theta(\widetilde{k}_{1},\ell_{2})}_{=P(Y_{2}=\ell_{2})}\cdot\lambda_{2}.
	\end{align*}
	Hence, we have shown for $\ell_{2}\in\left\{ 1,\ldots,b_{2}\right\} $
	\begin{align*}
	P(Y_{2}=\ell_{2})\cdot\lambda_{2} & =P(Y_{1}=b_{1}-\left(b_{2}-(\ell_{2}-1)\right),Y_{2}=\ell_{2}-1)\cdot\frac{1}{2}\nu\\
	& \phantomeq+P(Y_{1}>b_{1}-\left(b_{2}-(\ell_{2}-1)\right),Y_{2}=\ell_{2}-1)\cdot\nu\\
	& =P(Y_{1}=b_{1}-b_{2}+\ell_{2}-1,Y_{2}=\ell_{2}-1)\cdot\frac{1}{2}\nu\\
	& \phantomeq+P(Y_{1}>b_{1}-b_{2}+\ell_{2}-1,Y_{2}=\ell_{2}-1)\cdot\nu.
	\end{align*}
\end{proof}

%--------------------------------------------

\textbf{Acknowledgement:} The present paper contains some results of my PhD thesis \cite{otten:17}, written under the supervision of Hans Daduna. I am deeply grateful to
him for his support and advice. Furthermore, I thank Karsten Kruse for helpful discussions on the subject of the paper.

\bibliographystyle{plainnat}
\bibliography{sn-bibliography}
	
\end{document}